\documentclass[12pt, reqno]{amsart}
\usepackage{amsmath, amssymb, amscd, amsthm, comment, amsxtra}
\usepackage{graphicx}
\usepackage{subfigure}
\usepackage{enumerate}
\usepackage[pdftex, linktocpage=true]{hyperref}
\usepackage{euscript}
\usepackage[format=plain,labelfont=bf,up,width=.99\textwidth]{caption}
\usepackage[toc,page]{appendix}
\usepackage{xcolor}
\usepackage{soul}

\theoremstyle{plain}
\newtheorem{thm}{Theorem}[section]
\newtheorem{lem}[thm]{Lemma}
\newtheorem{prop}[thm]{Proposition}
\newtheorem{cor}[thm]{Corollary}
\newtheorem*{thma}{Stable Manifold Theorem}
\newtheorem*{thmb}{Center Manifold Theorem}

\theoremstyle{definition}
\newtheorem{defn}[thm]{Definition}

\newtheorem{rem}[thm]{Remark}

\newtheorem*{problem*}{Open Problem}

\numberwithin{equation}{section}
\newcommand{\thmref}[1]{Theorem~\ref{#1}}
\newcommand{\propref}[1]{Proposition~\ref{#1}}
\newcommand{\lemref}[1]{Lemma~\ref{#1}}
\newcommand{\corref}[1]{Corollary~\ref{#1}}

\newcommand{\secref}[1]{Section~\ref{#1}}
\newcommand{\subsecref}[1]{Subsection~\ref{#1}}
\newcommand{\appref}[1]{Appendix~\ref{#1}}
\newcommand{\remref}[1]{Remark~\ref{#1}}

\renewcommand{\epsilon}{\varepsilon}

\newcommand{\bbB}{\mathbb B}

\newcommand{\bbD}{\mathbb D}

\newcommand{\bbN}{\mathbb N}

\newcommand{\bbP}{\mathbb P}

\newcommand{\bbR}{\mathbb R}

\newcommand{\bbZ}{\mathbb Z}

\newcommand{\cC}{\mathcal C}

\newcommand{\cE}{\mathcal E}

\newcommand{\cK}{\mathcal K}
\newcommand{\cL}{\mathcal L}

\newcommand{\cN}{\mathcal N}

\newcommand{\cS}{\mathcal S}
\newcommand{\cT}{\mathcal T}
\newcommand{\cU}{\mathcal U}

\newcommand{\cZ}{\mathcal Z}

\newcommand{\bfC}{\mathbf C}

\newcommand{\hA}{\hat A}

\newcommand{\hE}{\hat E}
\newcommand{\hF}{\hat F}

\newcommand{\hK}{\hat K}

\newcommand{\hc}{\hat c}

\newcommand{\he}{\hat e}
\newcommand{\hf}{\hat f}

\newcommand{\tiA}{\tilde A}

\newcommand{\tiC}{\tilde C}

\newcommand{\tiE}{\tilde E}
\newcommand{\tiF}{\tilde F}

\newcommand{\tiU}{\tilde U}

\newcommand{\tie}{\tilde e}
\newcommand{\tif}{\tilde f}
\newcommand{\tig}{\tilde g}

\newcommand{\tiz}{\tilde z}

\newcommand{\hrho}{\hat \rho}
\newcommand{\hsigma}{\hat \sigma}

\newcommand{\homega}{\hat \omega}
\newcommand{\bepsilon}{{\bar \epsilon}}

\newcommand{\bL}{{\bar L}}

\newcommand{\bcS}{{\bar{\cS}}}

\newcommand{\hcS}{\hat \cS}

\newcommand{\tixi}{\tilde \xi}

\newcommand{\tirho}{\tilde \rho}

\newcommand{\chU}{\check U}

\newcommand{\uL}{\underline{L}}

\newcommand{\baeta}{\bar\eta}

\newcommand{\chalpha}{\check{\alpha}}

\newcommand{\chl}{\check{l}}

\newcommand{\ticS}{\tilde{\cS}}
\newcommand{\chA}{\check{A}}

\newcommand{\bA}{\bar{A}}
\newcommand{\chF}{\check{F}}
\newcommand{\chf}{\check{f}}
\newcommand{\chg}{\check{g}}
\newcommand{\frG}{\mathfrak{G}}
\newcommand{\frDF}{\mathfrak{DF}}
\newcommand{\chlambda}{\check{\lambda}}
\newcommand{\chbeta}{\check {\beta}}

\newcommand{\Jac}{\operatorname{Jac}}
\newcommand{\diam}{\operatorname{diam}}

\newcommand{\loc}{\operatorname{loc}}

\newcommand{\ecc}{\operatorname{ecc}}

\newcommand{\sgn}{\operatorname{sign}}

\newcommand{\matsp}[1]{\hspace{5mm} \text{#1} \hspace{5mm}}

\newcommand{\comma}{, \hspace{5mm}}

\title[Quantitative Estimates on  Invariant Manifolds]{Quantitative Estimates on  Invariant Manifolds for Surface Diffeomorphisms}
\author{Sylvain Crovisier, Mikhail Lyubich, Enrique Pujals, Jonguk Yang}
\thanks{The first author was partially supported by the ANR project CoSyDy. The second author was partly supported by the NSF, the Hagler and Clay Fellowships, the Institute for Theoretical Studies at ETH (Zurich), SLMSI (formerly MSRI Berkeley), the Fields Institute and the Center of Nonlinear Analysis and Modeling at the University of Toronto. The third author was partly supported by NSF award number 24000048. The fourth author was partially supported by SLMSI, Institut Mittag-Leffler, the thematic programs `Topological, smooth and holomorphic dynamics, ergodic theory, fractals’ of the Simons Foundation Award No. 663281 for IMPAN, the Fields Institute and the Center of Nonlinear Analysis and Modeling at the University of Toronto.}
\begin{document}

\maketitle

\begin{abstract}
We carry out a detailed quantitative analysis on the geometry of invariant manifolds for smooth dissipative systems in dimension two. We begin by quantifying the regularity of any  orbit (finite or infinite) in the phase space with a set of explicit inequalities. Then we relate this directly to the quasi-linearization of the local dynamics on regular neighborhoods of this orbit. The parameters of regularity explicitly determine the sizes of the regular neighborhoods and the smooth norms of the corresponding regular charts. As a corollary, we establish the existence of smooth stable and center manifolds with uniformly bounded geometries for regular orbits independently of any pre-existing invariant measure. This provides us with the technical background for the renormalization theory of H\'enon-like maps developed in \cite{CLPY1}, \cite{CLPY2} and \cite{Y}.
\end{abstract}

\tableofcontents

\section{Introduction}

In 1970's, Pesin established the foundations for the study of non-uniformly hyperbolic dynamical systems through a series of landmark works \cite{Pe1},\cite{Pe2},\cite{Pe3}. The classical implementation of this technique reveals statistical information about the system under consideration with respect to a given invariant measure. However, for many applications, it is important to apply the core mechanisms of Pesin theory without having a pre-established invariant measure, and, in particular, without having the asymptotic information provided by Oseledec's theorem. In this paper, we attempt to lay the groundwork for such a theory. A similar type of approach was developed in \cite{CP} and \cite{HL}.

Let $r\geq 2$ be an integer, and consider a $C^{r+1}$ surface diffeomorphism $F$. We begin by formulating the relevant regularity conditions that control dominating and hyperbolic features of any finite or infinite orbit. Consider a point $p$ in the surface and a tangent direction $E_p$ at $p$. To quantify the regularity of $p$, we introduce the following parameters: $\lambda, \rho, \epsilon \in (0,1)$; $L \geq 1$ and $N, M \in \bbN \cup \{0, \infty\}$. We say that the point $p$ is {\it $N$-times forward $(L, \epsilon, \lambda, \rho)_v$-regular along $E_p$} if, for all $1 \leq n \leq N$, we have
$$
L^{-1}\lambda^{(1+\epsilon)n}\leq \|DF^n|_{E_p}\| \leq L \lambda^{(1-\epsilon)n}
$$
and
$$
L^{-1}\rho^{(1+\epsilon)n}\leq \frac{\|DF^n|_{E_p}\|^2}{\Jac_pF^n} \leq L \rho^{(1-\epsilon)n}.
$$
Intuitively, the former condition gives exponential contraction along $E_p$, and the latter gives the domination of this contraction (there is exponential repulsion in the projective tangent space of other directions from $E_p$). Similarly, $p$ is {\it $M$-times backward $(L, \epsilon, \lambda, \rho)_v$-regular along $E_p \in \bbP^2_p$} if, for all $1 \leq n \leq M$, we have
$$
L^{-1}\lambda^{(1+\epsilon)n}\leq \|DF^{-n}|_{E_p}\|^{-1} \leq L \lambda^{(1-\epsilon)n}
$$
and
$$
L^{-1}\rho^{(1+\epsilon)n}\leq \frac{\Jac_pF^{-n}}{\|DF^{-n}|_{E_p}\|^2} \leq L \rho^{(1-\epsilon)n}.
$$


The main technical result of the paper is that if the following conditions hold on the parameters of regularity:
$$
\frac{\rho^{(r+1)-\epsilon(r+3)}}{\lambda^{(1+\epsilon)r}} < 1
\matsp{and}
\frac{\lambda^{(1-\epsilon)(r+1)}}{\rho^{r+\epsilon(2-r)}} < 1,
$$
then there exist $C^r$ changes-of-co\-or\-di\-na\-tes, called {\it regular charts}, defined on slow-exponentially shrinking neighborhoods covering a regular orbit, called {\it regular neighborhoods}, that conjugates the local dynamics of $F$ to a nearly linear one. Moreover, the conjugated maps have a certain skew-product form which is particularly useful for obtaining strong $C^r$-estimates for the dynamics. This construction is referred to as a {\it quasi-linearization} (or {\it Q-linearization}) {\it of $F$}, and is given in \secref{sec.chart}.

As consequences of Q-linearization, we obtain invariant curves with uniformly bounded local geometries that approximate the local stable and center manifolds in case of finite-time regularity, and coincide with them in case of infinite-time regularity. This can be loosely stated as follows (for the precise statements, see Theorems \ref{stable} and \ref{center jet}).

\begin{thma}
Suppose
$$
\frac{\lambda^{1-5\epsilon}}{\rho^{8\epsilon}} < 1 \matsp{and} \frac{\rho^{1-\epsilon}}{\lambda^{2\epsilon}} < 1.
$$
Assume that $p$ is infinite-time forward regular along $E_p$. Then $p$ has a unique $C^r$ strong-stable manifold $W^{ss}(p)$, which is tangent to $E_p$.

Additionally, let $W^{ss}_{\loc}(p) \ni p$ be the connected component of $W^{ss}(p)$ intersected with the regular neighborhood at $p$. Then the $C^r$-norm of the arc-length parameterization of $W^{ss}_{\loc}(p)$ is uniformly bounded in terms of the parameters of regularity of $p$.
\end{thma}

\begin{thmb}
Suppose
$$
\rho = \lambda
\matsp{and}
\epsilon < \frac{1}{11r+2}.
$$
Assume that $p$ is infinite-time backward regular along $E_p$. Then $p$ has a $C^r$ center manifold $W^c(p)$ contained in the regular neighborhood at $p$, whose tangent direction $E^c_p$ at $p$ is transverse to $E_p$. Its $C^r$-jet at $p$ is unique.

Furthermore, the $C^r$-norm of the arc-length parameterization of $W^c(p)$ is uniformly bounded in terms of the parameters of regularity of $p$.
\end{thmb}

\begin{rem}
In fact, we prove a stronger version of the Stable Manifold Theorem: namely that there exists a local stable foliation whose $C^r$-norm is uniformly bounded.
\end{rem}

The machinery developed in this paper plays an integral role in our subsequent works on the renormalization theory of dissipative H\'enon-like maps in \cite{CLPY1}, \cite{CLPY2} and \cite{Y}. In the class of 2D dynamical systems that we consider, each map features a unique orbit of tangencies between strong-stable and center manifolds. We refer to this orbit as the {\it critical orbit}.  Note that the critical orbit is highly atypical, and hence is unaccounted for by classical measurable Pesin theory. Nonetheless, it turns out that this single orbit is the primary source of non-linearity for the system, and is chiefly responsible for shaping its overall dynamics (analogously to the critical orbit of a unimodal map in 1D). Quantitative Pesin theory provides an adequate language to analyze and describe how this happens.

Although the aforementioned works are specifically about renormalization, the tools developed in this paper are quite general. We expect that they will find applications in other settings as well. In particular, we believe that they will be especially useful in the study of dissipative systems that feature tangencies between stable and center/unstable manifolds, in the spirit of the work of Benedicks and Carleson \cite{BeCa}. Furthermore, it should be commented that we restrict ourselves to the 2D case mainly for simplicity. Our quantitative approach should naturally generalize to arbitrary dimensions.

\section{Derivatives in Projective Space}

In this section we supply for reader's convenience some basic calculations for the first and second derivatives of the projectivization of a smooth two-dimensional map. We also define the notion of the projective attracting and repelling directions.

For $p \in \bbR^2$, denote the projective tangent space at $p$ by $\bbP^2_p$. In this section, we write an element of $\bbP^2_p$ as
$$
E_p^t := \{r(\cos t, \sin t) \; | \; r \in \bbR\}
\matsp{for}
t \in \bbR/2\pi\bbZ.
$$

Let $F : \Omega \to F(\Omega)$ be a $C^1$-diffeomorphism on a domain $\Omega \subset \bbR^2$. For $p \in \Omega$, define $l_p : \bbR/2\pi\bbZ \to \bbR^+$ and $\theta_p : \bbR/2\pi\bbZ \to \bbR/2\pi\bbZ$ by
$$
l_p(t)(\cos(\theta_p(t)), \sin(\theta_p(t))) := D_pF(\cos t, \sin t)
\matsp{for}
t \in \bbR/\bbZ.
$$
For $n \geq 1$, the {\it $n$th projective derivative of $F$ at $E_p^t$} is defined as
$$
\partial_\bbP^n F(E_p^t) := \theta_p^{(n)}(t).
$$
Also define the {\it growth variance of $F$ at $p$} as $\max_t|l_p'(t)|$. 

\begin{prop}\label{proj 1 deriv}
We have
$$
\partial_\bbP F(E_p^t) = \frac{\Jac_pF}{\|D_pF|_{E_p^t}\|^2}.
$$
\end{prop}

\begin{proof}
Let $v_{\max} = (\cos\alpha, \sin\alpha) \in \bbP^2_p$ be the direction of maximum expansion for $D_pF$. Let
$$
\beta = \theta_p(\alpha).
$$

Denote
$$
R_t = \begin{bmatrix}
\cos t & -\sin t \\
\sin t & \cos t
\end{bmatrix}.
$$
Consider
$$
A = \begin{bmatrix}
a & 0\\
0 & b
\end{bmatrix} = R_\beta^{-1} D_pF R_\alpha.
$$
Note that $a > b$.

For $v = (\cos t, \sin t) \in \bbP^2_p$, write
$$
w = D_pF(v) = (l\cos(\tau), l\sin(\tau)).
$$
Observe
$$
l(\cos(\tau - \beta), \sin(\tau - \beta)) = R_\beta^{-1}w =  A R_\alpha^{-1} v = (a\cos(t - \alpha), b\sin(t - \alpha)).
$$
So
\begin{equation}\label{eq:proj 1 deriv}
\tan(\tau-\beta) = \frac{b}{a}\tan(t - \alpha).
\end{equation}
Differentiating, we obtain
$$
\frac{d\tau}{dt} = \frac{b}{a} \left(\frac{\cos(\tau-\beta)}{\cos(t - \alpha)}\right)^2 = \frac{ab}{l^2}.
$$
\end{proof}

For $\Lambda \subset \Omega$, define the {\it eccentricity of $F$ on $\Lambda$} as
$$
\ecc_\Lambda(F) := \sup_{p \in \Lambda}\|D_pF\|\|D_{F(p)}F^{-1}\|.
$$

\begin{prop}\label{proj 2 deriv}
The first and second projective derivatives, and the growth variance of $F$ on $\Lambda$ are uniformly bounded in terms of $\ecc_\Lambda(F)$.
\end{prop}

\begin{proof}
Consider the same set up as in the proof of \propref{proj 1 deriv}. Clearly,
$$
\frac{d\tau}{dt} < \frac{a}{b} < \ecc_\Lambda(F).
$$

Differentiating \eqref{eq:proj 1 deriv} twice, we obtain
\begin{align*}
\frac{d^2\tau}{dt^2} &= \frac{2b}{l}\left(\frac{\frac{d\tau}{dt}\sin(\tau-\beta)\cos(t-\alpha)-\cos(\tau-\beta)\sin(t-\alpha)}{\cos^2(t-\alpha)}\right)\\
&= 2\frac{ab}{l^2}\tan(t-\alpha)\left(\frac{b^2}{l^2}-1\right)\\
&= \frac{2ab(b+l)\sin(t-\alpha)}{l^4}\frac{b-l}{\cos(t-\alpha)}.
\end{align*}
If $\cos(t-\alpha) > k$ for some uniform constant $k > 0$, then we have
$$
\frac{d^2\tau}{dt^2} < \frac{2}{k}\frac{a}{b^3} < K (\ecc_\Lambda(F))^3
$$
for some uniform constant $K \geq 1$.

Lastly, recall that
$$
l = \sqrt{a^2\cos^2(t-\alpha) + b^2\sin^2(t-\alpha)}.
$$
Plugging in, and taking limits as $t \to \alpha \pm \pi/2$, we see that
$$
\left.\frac{d^2\tau}{dt^2}\right|_{t = \alpha \pm \pi/2} = 0.
$$

For the growth variance, we see that
\begin{align*}
\|l_p'\| &= \sup_t \frac{d}{dt} \sqrt{a^2\cos^2(t) + b^2\sin^2(t)}\\
&= \sup_t \frac{(a^2+b^2)\cos t \sin t}{\sqrt{a^2\cos^2 t + b^2 \sin^2 t}}\\
& < \frac{a^2+b^2}{b}.
\end{align*}
\end{proof}

\begin{defn}\label{defn.proj reg}
Let $\rho, \epsilon \in (0,1)$; $L \geq 1$ and $N, M \in \bbN \cup \{0, \infty\}$. A tangent direction $E_p \in \bbP^2_p$ at a point $p \in \Omega$ is called
\begin{itemize}
\item a {\it $(L, \epsilon, \rho)$-regular projective attractor for the $N$-forward iterates of $F$} if
\begin{equation}\label{eq:for proj attract}
L^{-1}\rho^{(1+\epsilon)n} < \partial_\bbP DF^n(E_p) < L \rho^{(1-\epsilon)n}
\matsp{for}
0 \leq n \leq N;
\end{equation}
\item a {\it $(L, \epsilon, \rho)$-regular projective repeller for the $N$-forward iterates of $F$} if
\begin{equation}\label{eq:for proj repel}
L^{-1}\rho^{-(1-\epsilon)n} < \partial_\bbP DF^n(E_p) < L \rho^{-(1+\epsilon)n}
\matsp{for}
0 \leq n \leq N;
\end{equation}
\item a {\it $(L, \epsilon, \rho)$-regular projective attractor for the $M$-backward iterates of $F$} if
\begin{equation}\label{eq:back proj attract}
L^{-1}\rho^{(1+\epsilon)n} < \partial_\bbP DF^{-n}(E_p) < L \rho^{(1-\epsilon)n}
\matsp{for}
0 \leq n \leq M;
\end{equation}
\item a {\it $(L, \epsilon, \rho)$-regular projective repeller for the $M$-backward iterates of $F$} if
\begin{equation}\label{eq:back proj repel}
L^{-1}\rho^{-(1-\epsilon)n} < \partial_\bbP DF^{-n}(E_p) < L \rho^{-(1+\epsilon)n}
\matsp{for}
0 \leq n \leq M.
\end{equation}
\end{itemize}
\end{defn}


\section{Dynamics in Projective Space}

Denote the projective space of $\bbR^2$ by $\bbP^2$. In this section, we write an element of $\bbP^2$ as
$$
E^t := \{r(\cos t, \sin t) \; | \; r \in \bbR\}
\matsp{for}
t \in \bbR/2\pi\bbZ.
$$

\subsection{Projective attractor}

In this subsection, we consider an orbit with a projective attracting direction. We construct a transverse projective repelling direction, and estimate the growth/contraction rate of the derivative of the original map along this direction.

For $0 \leq n < N$ with $N \in \bbN\cup\{\infty\}$, let
$$
A_n = \begin{bmatrix}
a_n & c_n \\
0 & b_n
\end{bmatrix}
$$
be a linear transformation with $a_n, b_n >0$ and $c_n \in \bbR$. Suppose $\|A_n^{\pm 1}\| < \bfC$
for some uniform constant $\bfC \geq 1$. For $1 \leq i \leq N - n$, denote
$$
A_n^i := A_{n+i-1}\cdot \ldots \cdot A_n.
$$
Suppose that there exist constants $\rho, \epsilon \in (0, 1)$ and $L \geq 1$ such that
\begin{equation}\label{eq:proj attract}
L^{-1}\rho^{(1+\epsilon)n} < \partial_\bbP A_0^n(E^0) = \frac{b_0 \ldots b_{n-1}}{a_0 \ldots a_{n-1}} < L \rho^{(1-\epsilon)n}
\matsp{for}
0 \leq n \leq N.
\end{equation}

\begin{rem}
Condition \eqref{eq:proj attract} means that the fixed horizontal direction $E^0$ exponentially attracts nearby directions under $A_0^n$ for $n$ sufficiently large.
\end{rem}

Define
$$
\sigma_n := \frac{\rho^{(1-\epsilon)n}}{(b_0\ldots b_{n-1})/(a_0\ldots a_{n-1})}
\matsp{and}
\cS_n := \begin{bmatrix}
1 & 0\\
0 & \sigma_n
\end{bmatrix}.
$$

\begin{lem}\label{proj reg dominated}
For $0 \leq n \leq N$, we have $1 < \sigma_n < L\rho^{-2\epsilon n}$, and
$$
\tiA_n :=\cS_{n+1}\cdot A_n \cdot (\cS_n)^{-1}= \begin{bmatrix}
a_n &  c_n/\sigma_n\\
0 & \tirho a_n
\end{bmatrix}
\matsp{where}
\tirho := \rho^{1-\epsilon}.
$$
\end{lem}

\begin{lem}\label{uni vert dir}
Consider the genuine vertical cone
$$
\nabla^{gv}(\omega) := \{(x,y) \; | \; x/y < \omega\}
\matsp{with}
\omega := \frac{\bfC^2}{\tirho(1-\tirho)}.
$$
Then we have
$$
\tiA_n^{-1}\left(\nabla^{gv}(\omega)\right) \Subset \nabla^{gv}(\omega).
$$
Consequently, if $N=\infty$, then there exists a unique direction $\hE_0 \in \bbP^2$ such that
$$
\hE_n = \tiA_{n-1} \cdot \ldots\cdot\tiA_0(\hE_0) \in \nabla^{gv}(\omega)
\matsp{for all}
n \geq 0.
$$
\end{lem}

\begin{proof}
Observe that
$$
\tiA_n^{-1} = \begin{bmatrix}
a_n^{-1} & -c_n /(\sigma_n\tirho a_n^2)\\
0 & \tirho^{-1}a_n^{-1}
\end{bmatrix}.
$$
Let $v = (x,y)$. Then
$$
\tiA_n^{-1}v = (x', y') = (a_n^{-1}x -c_n y /(\sigma_n\tirho a_n^2), y/(\tirho a_n)).
$$
We compute
$$
\left|\frac{x'}{y'}\right| = \left|\left(\tirho -\frac{c_n}{\tirho a_n \sigma_n}\cdot \frac{y}{x}\right) \cdot \frac{x}{y}\right|.
$$
If
$$
\omega > \frac{|c_n|}{\tirho(1-\tirho)a_n}  > \frac{|c_n|}{\tirho(1-\tirho)a_n \sigma_n},
$$
then it follows that $|x/y| > \omega$ implies $|x'/y'| < |x/y|$.
\end{proof}

\begin{prop}\label{dominant transverse}
There exists a direction $\hE \in \bbP^2$ such that
$$
\frac{1}{L^2\rho^{2\epsilon n}\sqrt{1+\omega^{2}}}< \frac{\|A_0^n|_{\hE}\|}{b_0 \ldots b_{n-1}} < L\rho^{-2\epsilon n}\sqrt{1+\omega^{2}}
\matsp{for all}
1 \leq n \leq N,
$$
where $\omega > 0$ is the uniform constant given in \lemref{uni vert dir}. If $N = \infty$, then $\hE$ is unique.
\end{prop}

\begin{proof}
If $N = \infty$, let $\hE_0$ be given by \lemref{uni vert dir}. Otherwise, let
$$
\hE_0 := (\tiA_{N-1} \cdot \ldots\cdot\tiA_0)^{-1}(E^{\pi/2}).
$$
Let $v_0 = (x_0, y_0) \in \hE_0$ with $y_0=1$, and denote $v_n = (x_n, y_n) := \tiA_0^n(v_0).$ By \lemref{uni vert dir}, we have $v_n \in \nabla^{gv}(\omega)$ for $0 \leq n \leq N$. Thus,
$$
|y_n| \leq \|v_n\| \leq \sqrt{1+\omega^{2}} \cdot |y_n|.
$$
Hence,
$$
1\leq \frac{\|\tiA_0^n|_{\hE_0}\|\|v_0\|}{\rho^{(1-\epsilon)n}a_0 \ldots a_{n-1}} \leq \sqrt{1+\omega^{2}}.
$$
Observe that
$$
1 < \frac{\|\tiA_0^n|_{\hE_0}\|}{\|A_0^n|_{\hE_0}\|} < \sigma_n.
$$
Moreover,
$$
\rho^{(1-\epsilon)n}  a_0 \ldots a_{n-1} = \left(\frac{\rho^{(1-\epsilon)n}}{(b_0 \ldots b_{n-1})/(a_0 \ldots a_{n-1})}\right) \cdot b_0 \ldots b_{n-1},
$$
and the term in the bracket is bounded by
$$
L^{-1}\leq \frac{\rho^{(1-\epsilon)n}}{(b_0 \ldots b_{n-1})/(a_0 \ldots a_{n-1})} \leq L\rho^{-2\epsilon n} .
$$
So,
$$
\left(L\sigma_n \sqrt{1+\omega^{2}}\right)^{-1} \leq \frac{\|A_0^n|_{\hE_0}\|}{b_0 \ldots b_{n-1}} \leq L\rho^{-2\epsilon n}\sqrt{1+\omega^{2}}.
$$
The result now follows from \lemref{proj reg dominated}.
\end{proof}

\begin{rem}
Intuitively, \propref{dominant transverse} means that the existence of a ``projective attractor'' $E^0$ implies the existence of a transverse ``projective repeller'' $\hE$ (which is unique if $N=\infty$).
\end{rem}

\subsection{Projective repeller}\label{subsec.proj rep}

This subsection is dual to the previous one. Assuming the existence of a projective repelling direction, we establish that the orbit of any other direction is asymptotic to a projective attractor, and estimate the corresponding rate of convergence.

For $0 \leq n < N$ with $N \in \bbN\cup\{\infty\}$, let
$$
A_n = \begin{bmatrix}
a_n & 0 \\
c_n & b_n
\end{bmatrix}
$$
be a linear transformation with $a_n, b_n >0$ and $c_n \in \bbR$. Suppose $\|A_n\| < \bfC$ for some uniform constant $\bfC \geq 1$. Additionally, suppose that there exist constants $\rho, \epsilon \in (0,1)$ and $L \geq 1$ such that
\begin{equation}\label{eq:proj repel}
L^{-1} \rho^{-(1-\epsilon)n}< \partial_\bbP A_0^n(E^{\pi/2}) = \frac{a_0 \ldots a_{n-1}}{b_0 \ldots b_{n-1}} < L\rho^{-(1+\epsilon)n}
\matsp{for}
0 \leq n \leq N.
\end{equation}

\begin{rem}
Condition \eqref{eq:proj repel} means that the fixed vertical direction $E^{\pi/2}$ exponentially repels nearby directions under $A_0^n$ for sufficiently large $n$.
\end{rem}

Define
$$
\hsigma_n := \frac{(b_0\ldots b_{n-1})/(a_0\ldots a_{n-1})}{\rho^{(1+\epsilon)n}}
\matsp{and}
\ticS_n := \begin{bmatrix}
\hsigma_n & 0\\
0 & 1
\end{bmatrix}.
$$

\begin{lem}\label{uni hor dir}
For $1 \leq n \leq N$, we have $1 < \hsigma_n < L\rho^{-2\epsilon n}$, and
$$
\hA_n :=\hcS_{n+1}\cdot A_n \cdot (\hcS_n)^{-1}= \begin{bmatrix}
b_n/\hrho & 0\\
c_n/\hsigma_n & b_n
\end{bmatrix},
\matsp{where}
\hrho:= \rho^{1+\epsilon}.
$$
\end{lem}

\begin{prop}\label{proj hor dir}
Consider the genuine horizontal cone 
$$
\nabla^{gh}(\homega) := \{(x,y) \; | \; y/x < \homega\}
\matsp{where}
\homega:=\frac{\hrho}{1-\hrho}\cdot \bfC^2.
$$
Then
$$
\hA\left(\nabla^{gh}(\homega)\right) \Subset \nabla^{gh}(\homega).
$$
\end{prop}

\begin{proof}
Let $v = (x,y)$. Then
$$
\hA_nv = (x', y') = (b_n x/\hrho, c_nx/\hsigma_n + b_ny).
$$
We compute
$$
\left|\frac{y'}{x'}\right| = \left|\hrho\left(1 +\frac{c_n}{ b_n\hsigma_n}\cdot \frac{x}{y}\right) \cdot \frac{y}{x}\right|.
$$
If
$$
\homega > \frac{\hrho}{1-\hrho} \cdot \frac{|c_n|}{b_n} > \frac{\hrho}{1-\hrho} \cdot \frac{|c_n|}{b_n \hsigma_n}.
$$
then it follows that $|y/x| > \homega$ implies $|y'/x'| < |y/x|$.
\end{proof}

For $1 \leq n \leq N$, let
$$
\cT_n = \begin{bmatrix}
1 & 0\\
-\tau_n& 1
\end{bmatrix}
$$
be a linear transformation that maps $\tiE_n := \tiA_0^n(E^0)$ to $E^0$. Note that $|\tau_n| \leq \homega$. Moreover, we have
$$
\chA_n= \begin{bmatrix}
b_n/\rho^{1+\epsilon} & 0\\
0 & b_n
\end{bmatrix}
=\cT_{n+1}\cdot \hA_n \cdot (\cT_n)^{-1}.
$$

\begin{prop}\label{dominated transverse}
For $t \in (0, \pi/2]$, we have
$$
\frac{1}{KL|t|^2}\rho^{(1+3\epsilon) n} < \partial_\bbP A_0^n(E^{\pi/2 -t}) < \frac{KL}{|t|^2}\rho^{(1-\epsilon) n}
\matsp{for all}
1\leq n \leq N,
$$
where $K = K(\homega) \geq 1$ is a uniform constant.
\end{prop}

\begin{proof}
Let $v_0 = (x_0, y_0) \in E^{\pi/2-t}$ such that $\|v_0\| =1$. Denote
$$
E^{s_n} := \chA_0^n(E^{\pi/2-t})
\matsp{and}
v_n = (x_n, y_n) := \chA_0^n(v_0) \in E^{s_n}.
$$
Let $0 \leq m \leq N$ be the largest number such that $s_m \leq \pi/4$. Note that
\begin{equation}\label{eq.domin trans 1}
\rho^{(1+\epsilon)m} \asymp |t|.
\end{equation}
A straightforward computation shows that
\begin{equation}\label{eq.domin trans 2}
\|v_n\| \asymp |b_0 \cdot \ldots \cdot b_{n-1}|
\matsp{for}
n \leq m.
\end{equation}
Similarly, we have
\begin{equation}\label{eq.domin trans 3}
\|v_n\| \asymp \frac{|b_0 \cdot \ldots \cdot b_{n-1}|}{\rho^{(1+\epsilon)(n-m)}}
\matsp{for}
m \leq n \leq N.
\end{equation}

First we establish upper bounds. For $n \leq m$:
\begin{align*}
\partial_\bbP A_0^n(E^{s_0}) &= \partial_\bbP(\hcS_n^{-1} \cdot \cT_n^{-1} \cdot \chA_0^n \cdot \cT_0 \cdot \hcS_0)(E^{s_0})\\
&< K \partial_\bbP\chA_0^n(E^{s_0})\\
&= K \frac{\Jac \chA_0^n}{\|v_n\|^2}\\
&\asymp K \rho^{-(1+\epsilon)n}\\
&< K |t|^{-2}\rho^{(1+\epsilon)n}.
\end{align*}
For $n \geq m$:
\begin{align*}
\partial_\bbP A_0^n(E^{s_0}) &< K\hsigma_n \partial_\bbP\chA_0^n(E^{s_0})\\
&= K\hsigma_n \frac{\Jac \chA_0^n}{\|v_n\|^2}\\
&\asymp K\hsigma_n \frac{\rho^{2(1+\epsilon)(n-m)}}{\rho^{(1+\epsilon)n}}\\
&\asymp K\hsigma_n |t|^{-2} \rho^{(1+\epsilon)n}\\
&< KL|t|^{-2}\rho^{(1-\epsilon)n}.
\end{align*}

Next, we establish lower bounds. For $n \leq m$:
\begin{align*}
\partial_\bbP A_0^n(E^{s_0}) &> \frac{1}{K\hsigma_n} \partial_\bbP\chA_0^n(E^{s_0})\\
&\asymp  \frac{1}{K\hsigma_n} \rho^{-(1+\epsilon)n}\\
&> \frac{1}{KL}\rho^{(-1+\epsilon)n}.
\end{align*}
For $n \geq m$:
\begin{align*}
\partial_\bbP A_0^n(E^{s_0}) &>  \frac{1}{K\hsigma_n} \partial_\bbP\chA_0^n(E^{s_0})\\
&\asymp \frac{1}{K\hsigma_n} \frac{\rho^{2(1+\epsilon)(n-m)}}{\rho^{(1+\epsilon)n}}\\
&>\frac{1}{KL}\rho^{(1+3\epsilon)n}.
\end{align*}
\end{proof}

\begin{rem}
Intuitively, \propref{dominated transverse} means that the existence of a``projective repeller'' $E^{\pi/2}$ implies that the height of any horizontal cone in the complement contracts exponentially fast.
\end{rem}


\section{Dynamics of Almost Linear Maps}

In this section, we consider a bi-infinite sequence of global diffeomorphisms of $\bbR^2$ that are $C^1$-close to diagonal linear maps satisfying domination condition with strongly attracting vertical direction. Moreover, we assume uniform bounds on the $C^r$-norms of these maps. Under these circumstances, using the $C^r$-section theorem, we construct:
\begin{itemize}
\item a sequence of invariant $C^r$ horizontal graphs that, in forward time, attracts exponentially fast in $C^r$-topology any sufficiently horizontal graph; and dually
\item a sequence of global $C^{r-1}$ vertical direction fields that, in backward time, attracts exponentially fast in $C^{r-1}$-topology any sufficiently vertical direction field.
\end{itemize}
We follow up with $C^r$-bounds for compositions of these maps.

Fix an integer $r \geq 1$. Let $\lambda, \eta, \rho \in (0,1)$; $0 < \mu_- < \mu_+$ and $\bfC \geq 1$ be constants such that $\eta < \lambda$ and $\lambda/\mu_- < \rho$. For $m \in \bbZ$, let $F_m : (\bbR^2, 0) \to (\bbR^2, 0)$ be a $C^r$-diffeomorphism such that
$$
D_0F_m = \begin{bmatrix}
a_m & 0\\
0 & b_m
\end{bmatrix}
$$
for some $a_m, b_m > 0$, and the following properties are satisfied:
\begin{itemize}
\item $\mu_-+\eta < a_m< \mu_+-\eta$ and $b_m < \lambda -\eta$;
\item $\sup_p\|D_pF_m - D_0F_m\| < \eta/2$; and
\item $\|DF_m\|_{C^{r-1}} < \bfC$.
\end{itemize}
For $n \in \bbN$, denote $F_m^n := F_{m+n-1} \circ \ldots \circ F_m$.

\subsection{Applications of $C^r$-Section Theorem}\label{subsec.cr sec}

In this section, we formulate some classical applications of graph transform techniques and $C^r$-section theorem. See \cite[Chapters 5 and 6]{Sh} for more details.

Let $g : \bbR \to \bbR$ be a $C^r$-function. The curve
$$
\Gamma_g := \{(x, g(x)) \; x \in \bbR\}
$$
is the {\it horizontal graph of $g$}. For $t \geq 0$, we say that $\Gamma_g$ is {\it $t$-horizontal} if $\|g'\| \leq t$. Additionally, $\Gamma_g$ is {\it center-aligned} if $g(0) = 0$ and $g'(0)=0$. The space of center-aligned $t$-horizontal graphs of $C^r$-functions is denoted $\frG_h^r(t)$. We define a metric on $\frG_h^r(t)$ by
$$
\|\phi_1 - \phi_2\|_* := \sup_{x \neq 0} \frac{\|\phi_1(x)-\phi_2(x)\|}{\|x\|}.
$$

Let $H : \bbR^2 \to \bbR^2$ be a $C^r$-diffeomorphism. Suppose that there exists a $C^r$-function $H_*(g) : \bbR \to \bbR$ such that $H(\Gamma_g) = \Gamma_{H_*(g)} \in \frG_h^r$. Then $H_*(g)$ and $\Gamma_{H_*(g)}$ are referred to as the {\it horizontal graph transform of $g$} and $\Gamma_g$ {\it by $H$} respectively.

\begin{prop}[Forward horizontal graph transform]\label{for gt}
Suppose
\begin{equation}\label{eq.for cond}
\frac{\lambda}{\mu_-^r} < 1.
\end{equation}
Then there exists a uniform constant $\omega = \omega(\rho, \eta) >0$ with $\omega(\rho, \eta) \to 0$ as $\eta \to 0$ such that the following holds for all $m \in \bbZ$. For $\Gamma_g \in \frG^r_h(\omega)$, the horizontal graph transform $(F_m)_*(g)$ is well-defined. Moreover, there exists a unique sequence $\{\Gamma_{g^*_m}\}_{m\in\bbZ} \subset \frG^r_h(\omega)$ such that $F_m(\Gamma_{g^*_m}) = \Gamma_{g^*_{m+1}}$, and
for $\Gamma_g \in \frG_h^1(t)$, we have
$$
\|(F_m^n)_*(g) - g^*_{m+n}\|_* < \left(\frac{\lambda}{\mu_-}\right)^n \|g- g^*_m\|_*
\matsp{for}
n \in \bbN.
$$
Additionally, $\|(g_m^*)''\|_{C^{r-2}} < K$ for some uniform constant $K = K(\bfC, \lambda, \mu_-, r) >0$.
\end{prop}

For $p \in \bbR^2$ and $u \in \bbR$, let $E_p^u \in \bbP_p^2$ be the tangent direction at $p$ given by
$$
E_p^u := \{r(u, 1) \; | \; r \in \bbR\}.
$$
Let $\xi : \bbR^2 \to \bbR$ be a $C^{r-1}$-map. The direction field
$$
\cE_\xi := \{E_p^{\xi(p)} \; | \; p \in \bbR^2\}
$$
is the {\it vertical direction field of $\xi$}. For $t \geq 0$, we say that $\cE_\xi$ is {\it $t$-vertical} if $\|\xi\| \leq t$. The space of $t$-vertical direction fields of $C^{r-1}$-maps is denoted $\frDF_v^{r-1}(t)$. 

Let $H : \bbR^2 \to \bbR^2$ be a $C^r$-diffeomorphism. Suppose that there exists a $C^{r-1}$-map $H^*(\xi) : \bbR^2 \to \bbR$ such that $DH^{-1}(\cE_\xi) =\cE_{H^*(\xi)} \in \frDF_v^{r-1}$. Then $H^*(\xi)$ and $\cE_{H^*(\xi)}$ are referred to as the {\it vertical direction field transform of $\xi$} and $\cE_\xi$ {\it by $H$} respectively.

\begin{prop}[Backward vertical direction field transform]\label{back dt}
Suppose
\begin{equation}\label{eq.back cond}
\lambda\mu_+^{r-1} <1.
\end{equation}
Then there exists a uniform constant $\omega = \omega(\rho, \eta) >0$ with $\omega(\rho, \eta) \to 0$ as $\eta \to 0$ such that the following holds for all $m\in\bbZ$. For $\cE_\xi \in \frDF^0_v(\omega)$, the vertical direction field transform $(F_m)^*(\xi)$ is well-defined. Moreover, there exists a unique sequence $\{\cE_{\xi^*_m}\}_{m\in\bbZ} \subset \frDF^{r-1}_v(\omega)$ such that $DF_m^{-1}(\cE_{\xi^*_m}) = \cE_{\xi^*_{m-1}}$, and for $\cE_\xi \in \frDF^0_v(\omega)$, we have
$$
\|(DF_{m-n}^n)^*(\xi) - \xi^*_{m-n}\|_{C^0} < (\lambda\mu_+)^n \|\xi - \xi^*_m\|_{C^0}
\matsp{for}
n \in \bbN.
$$
Additionally, $\|D\xi^*_m\|_{C^{r-2}} < K$ for some uniform constant $K = K(\bfC, \lambda, \mu_+, r) >0$.
\end{prop}

\begin{prop}[Rectification]\label{rectify}
Let $\omega \in (0, 1/2)$ and $K > 0$. Consider $\Gamma_g \in \frG_h^r(\omega)$ and $\cE_\xi \in \frDF_v^{r-1}(\omega)$ such that
\begin{itemize}
\item $g(0) = 0$, $\|g'\| < \omega$ and $\|g''\|_{C^{r-2}} < K$; and
\item $\|\xi\| < \omega$ and $\|D\xi\|_{C^{r-2}} < K$.
\end{itemize}
Then there exists a unique $C^{r-1}$-chart $\Psi : (\bbR^2, 0) \to (\bbR^2, 0)$ such that
\begin{itemize}
\item $\Psi(x,g(x)) = (x,0)$ for $x \in \bbR$; and
\item $D\Psi(\cE_\xi(p)) = E_{\Psi(p)}^0$.
\end{itemize}
Moreover, we have $\|D\Psi\|_{C^{r-2}} < C$ for some uniform constant $C = C(K) >0$.
\end{prop}

\begin{proof}
Define
\begin{equation}\label{eq.hor align}
V(x,y) := (x, y-g(x)).
\end{equation}
For $z = (x,y) \in \bbR^2$, observe that
$$
D_zV(\xi(z), 1) = (\xi(z), 1 - g'(x)\xi(z)).
$$
Define
\begin{equation}\label{eq.foliate}
\tixi(z) := \frac{\xi(V^{-1}(z))}{1-g'(x)\xi(V^{-1}(z))}
\matsp{and}
H(x,y) := \left(x+\int_0^y \tixi(x,t)dt, y\right).
\end{equation}
Then $\Psi := H \circ V$ gives the desired rectifying map.
\end{proof}

\begin{lem}
For $m \in \bbZ$, let $\Psi_m$ be the chart given by \propref{rectify}, where $g$ and $\xi$ are taken to be $g_m^*$ and $\xi_m^*$ given by Propositions \ref{for gt} and \ref{back dt} respectively. Then $\tiF_m := \Psi_m \circ F_m \circ (\Psi_m)^{-1}$ is of the form
$$
\tiF_m(x,y) = (\tif_m(x), \tie_m(x,y))
\matsp{for}
(x,y) \in \bbR^2.
$$
Moreover, we have $|\partial_x^{r-1} \tie_m(x,y)| < K|y|$ for $y \in \bbR$, where $K = K(\bfC, \lambda, \mu_\pm, r) >0$ is a uniform constant.
\end{lem}

\begin{proof}
The first claim is obvious.

Write $\Psi_m = H_m \circ V_m$, where $V_m$ and $H_m$ are given by \eqref{eq.hor align} and \eqref{eq.foliate} respectively. Let
$$
\hF_m(x,y) = (\hf_m(x), \he_m(x,y)):= V_m \circ F_m \circ (V_m)^{-1}(x,y).
$$
Then $\hF_m$ is a $C^r$-diffeomorphism with $\|\hF_m\|_{C^r} < \hK$ for some uniform constant $\hK = \hK(\bfC, \mu_-)$. It follows immediately that for $0 \leq s < r$, we have
$$
|\partial_x^s \he_m(\cdot, y)| < \hK|y|
\matsp{for}
y \in \bbR.
$$

The first component of $H_m$ is given by
$$
h_m(x,y) = h_{m, y}(x) = x + \int_0^y \tixi_m(x,t)dt,
$$
with $\|\tixi_m\|_{C^{r-1}} < C$ for some uniform constant $C = C(\bfC) >0$. Hence, for $2 \leq s < r$, we have
$$
|\partial_x^s h_{m,y}| < C|y|
\matsp{for}
y \in \bbR.
$$

Observe that
$$
\tie_m(x,y) = \he_m(h_{m,y}^{-1}(x), y).
$$
The bound on $|\partial_x^{r-1} \tie_m(\cdot, y)|$ follows.
\end{proof}

\subsection{$C^r$-convergence under graph transform}

Suppose that $F_m$ for $m\in\bbZ$ is of the form
$$
F_m(x,y) = (f_m(x), e_m(x,y))
\matsp{for}
(x,y)\in\bbR^2
$$
where $f_m : \bbR \to \bbR$ is a $C^r$-diffeomorphism, and $e_m : \bbR^2\to \bbR$ is a $C^r$-map such that for all $0 \leq s \leq r$, we have
\begin{equation}\label{eq.decay}
|\partial_x^s e_m (\cdot, y)| < \bfC |y|
\matsp{for}
y \in \bbR.
\end{equation}
Clearly, we have
$$
\mu_- < f_m'(x) < \mu_+
\matsp{and}
\partial_y e_m(x,y) < \lambda < 1.
$$
Let $\sigma_- < 1$ and $\sigma_+ > 1$ be constants such that $\sigma_- \leq \mu_-$ and $\sigma_+ \geq \mu_+$. For $n \in \bbN$, denote
$$
F^n = (f^n, e^n) := F_{n-1} \circ \ldots \circ F_0.
$$

\begin{prop}[Convergence of horizontal graphs]
Let $g : \bbR \to \bbR$ be a $C^r$-map with $\|g'\|_{C^{r-1}} < \infty$. For $m \in \bbZ$ and $n\in\bbN$, consider the graph transform $\tig := (F^n)_*(g)$. Then
$$
\|\tig'\|_{C^{r-1}} < C\left(\frac{\sigma_+}{\sigma_-}\right)^{(2r-1)n}\lambda^n(1+\|g'\|_{C^{r-1}})^r
$$
where $C = C(\bfC, \sigma_\pm, \lambda, r) \geq 1$ is a uniform constant.
\end{prop}

\begin{proof}
Observe that
$$
\tig(x) = e^n(u, g(u))
\matsp{where}
u := (f^n)^{-1}(x).
$$
The result follows from Propositions \ref{cr inverse}, \ref{cr hor} and \ref{sqze}.
\end{proof}

\begin{prop}[Convergence of vertical direction fields]
There exist uniform constants $C, \tiC \geq 1$ depending only on $\bfC, \sigma_+, \lambda, r$  such that the following holds. Let $\xi : \bbR^2 \to \bbR$ be a $C^{r-1}$-map with $\|\xi\|_{C^{r-1}} < \infty$. For $m \in \bbZ$ and $n\in\bbN$, consider the vertical direction transform 
$$
\tixi := (F^n)^*(\xi)|_{\bbR \times (-1, 1)}.
$$
Suppose
$$
C\sigma_+^{rn}\lambda^n(1+\|\xi\|_{C^{r-1}}) < \sigma_-^n.
$$
Then
$$
\|\tixi\|_{C^{r-1}} < \tiC\left(\frac{\sigma_+^3}{\sigma_-}\right)^{(r-1)n}\lambda^n\|\xi\|_{C^{r-1}}.
$$
\end{prop}

\begin{proof}
A straightforward computation shows that
$$
\tixi = \left(\frac{\xi \cdot \partial_y e^n}{\partial_x f^n - \xi \cdot \partial_x e^n}\right)\circ F^{n-1}.
$$
Denote
$$
\phi := \xi \cdot \partial_y e^n
\matsp{and}
\psi := \partial_x f^n - \xi \cdot \partial_x e^n.
$$
Then we can write
$$
\tixi = \phi \cdot (1/\psi) \circ F^{n-1}.
$$

In the following discussion, $C_i$ with $i \in \bbN$ denotes a uniform constant. By \lemref{cr product} and \propref{sqze}, we have
$$
\|\phi\|_{C^{r-1}} < C_1 \sigma_+^{rn}\lambda^n \|\xi\|_{C^{r-1}}.
$$
Similarly, \propref{cr hor} implies
$$
\|\psi\|_{C^{r-1}} < C_2\sigma_+^n + C_1\sigma_+^{rn}\lambda^n \|\xi\|_{C^{r-1}} < C_3 \sigma_+^n.
$$
Additionally, $\psi(x, y) > \sigma_-^n/2$ for all $x \in \bbR$ and $y \in (-1, 1)$. Lastly,
$$
\|DF^{n-1}\|_{C^{r-2}} < C_2\sigma_+^{n-1} + C_1\sigma_+^{(r-1)(n-1)}\lambda^{n-1} < C_4 \sigma_+^n.
$$
The result now follows from Lemmas \ref{hocb}, \ref{cr product} and \ref{cr quotient}.
\end{proof}


\section{Definitions of Regularity}\label{sec.defn reg}

In this section we introduce the notion of regularity, which quantifies the idea of having certain a rate of contraction/expansion in certain directions as well as domination relation between them. We give two definitions, one with respect to the strongly contracting/expanding direction, and another with respect to a transverse direction. Lastly, we show that these definitions are equivalent (up to uniformly increasing the size of the associated errors). This lays down technical background for the quasi-linearization results in the next section.

Consider a $C^1$-diffeomorphism $F : \Omega \to F(\Omega)$ defined on a domain $\Omega \subset \bbR^2$. Let $\lambda, \rho, \epsilon \in (0,1)$; $L \geq 1$ and $N, M \in \bbN \cup \{0, \infty\}$.  A point $p \in \Omega$ is {\it $N$-times forward $(L, \epsilon, \lambda, \rho)_v$-regular along $E_p^{v,+} \in \bbP^2_p$} if, for all $1 \leq n \leq N$, we have
\begin{equation}\label{eq:for reg 1}
L^{-1}\lambda^{(1+\epsilon)n}\leq \|DF^n|_{E^{v,+}_p}\| \leq L \lambda^{(1-\epsilon)n}
\end{equation}
and
\begin{equation}\label{eq:for reg 2}
L^{-1}\rho^{(1+\epsilon)n}\leq \frac{\|DF^n|_{E^{v,+}_p}\|^2}{\Jac_pF^n} \leq L \rho^{(1-\epsilon)n}.
\end{equation}
Similarly, $p$ is {\it $M$-times backward $(L, \epsilon, \lambda, \rho)_v$-regular along $E^{v,-}_p \in \bbP^2_p$} if, for all $1 \leq n \leq M$, we have
\begin{equation}\label{eq:back reg 1}
L^{-1}\lambda^{(1+\epsilon)n}\leq \|DF^{-n}|_{E^{v,-}_p}\|^{-1} \leq L \lambda^{(1-\epsilon)n}
\end{equation}
and
\begin{equation}\label{eq:back reg 2}
L^{-1}\rho^{(1+\epsilon)n}\leq \frac{\Jac_pF^{-n}}{\|DF^{-n}|_{E^{v,-}_p}\|^2} \leq L \rho^{(1-\epsilon)n}.
\end{equation}
If all four conditions \eqref{eq:for reg 1} - \eqref{eq:back reg 2} hold with $E_p^v := E_p^{v,+} = E_p^{v,-}$, then $p$ is {\it $(M, N)$-times $(L, \epsilon, \lambda, \rho)_v$-regular along $E_p^v$}. If, additionally, we have $M = N =\infty$, then $p$ is {\it Pesin $(L, \epsilon, \lambda, \rho)_v$-regular along $E_p^v$} .

We say that $p$ is {\it $N$-times forward $(L, \epsilon, \lambda, \rho)_h$-regular along $E_p^{h,+} \in \bbP_p^2$} if, for all $1 \leq n \leq N$, we have
\begin{equation}\label{eq:hor for reg 1}
L^{-1}\lambda^{(1+\epsilon)n} \leq \frac{\Jac_p F^n}{\|D_pF^n|_{E_p^{h,+}}\|} \leq L\lambda^{(1-\epsilon)n}
\end{equation}
and
\begin{equation}\label{eq:hor for reg 2}
L^{-1}\rho^{(1+\epsilon)n}  \leq \frac{\Jac_p F^n}{\|D_pF^n|_{E_p^{h,+}}\|^2} \leq L\rho^{(1-\epsilon)n}
\end{equation}
Similarly, we say that $p$ is {\it $M$-times backward $(L, \epsilon, \lambda, \rho)_h$-regular along $E_p^{h,-} \in \bbP_p^2$} if, for all $1 \leq n \leq M$, we have
\begin{equation}\label{eq:hor back reg 1}
L^{-1}\lambda^{(1+\epsilon)n}  \leq  \frac{\|D_pF^{-n}|_{E_p^{h,-}}\|}{\Jac_p F^{-n}} \leq L\lambda^{(1-\epsilon)n}
\end{equation}
and
\begin{equation}\label{eq:hor back reg 2}
L^{-1}\rho^{(1+\epsilon)n}  \leq  \frac{\|D_pF^{-n}|_{E_p^{h,-}}\|^2}{\Jac_p F^{-n}} \leq L\rho^{(1-\epsilon)n}.
\end{equation}
If all four conditions \eqref{eq:hor for reg 1} - \eqref{eq:hor back reg 2} hold with $E_p^h := E_p^{h,+} = E_p^{h,-}$, then $p$ is {\it $(M, N)$-times $(L, \epsilon, \lambda, \rho)_h$-regular along $E_p^h$}. If, additionally, we have $M = N =\infty$, then $p$ is {\it Pesin $(L, \epsilon, \lambda, \rho)_h$-regular along $E_p^h$}.

In the above definitions, the letters $v$ and $h$ stand for ``vertical'' and ``horizontal.'' The constants $L$, $\epsilon$, and $\lambda$ and $\rho$ are referred to as an {\it irregularity factor}, a {\it marginal exponent} and {\it contraction bases} respectively. Henceforth, once the contraction bases are introduced and fixed, we will sometimes write ``$(L, \epsilon, \lambda, \rho)_{v/h}$-regular'' as simply ``$(L, \epsilon)_{v/h}$-regular''.

\begin{rem}\label{rem.reg is proj reg}
Note that if $p$ is $(M, N)$-times $(L, \epsilon, \lambda, \rho)_v$-regular along $E_p^v$, then $E_p^v$ is a $(L, \epsilon, \rho)$-regular projective repeller and attractor for the $N$-forward iterates and $M$-backward iterates of $F$ respectively. Similarly, if $p$ is $(M, N)$-times $(L, \epsilon, \lambda, \rho)_h$-regular along $E_p^h$, then $E_p^h$ is a $(L, \epsilon, \rho)$-regular projective attractor and repeller for the $N$-forward iterates and $M$-backward iterates of $F$ respectively. 
\end{rem}

\begin{prop}[Vertical forward regularity $=$ horizontal forward regularity]\label{transverse for reg}
There exists a uniform constant $K = K(\rho, \|F\|_{C^1}) \geq 1$ such that the following holds. Suppose $p$ is $N$-times forward $(L, \epsilon)_v$-regular along $E_p^v \in \bbP_p^2$. Let $E_p^h \in \bbP_p^2 \setminus \{E_p^v\}$. If $\measuredangle (E_p^v, E_p^h) > \theta$, then the point $p$ is $N$-times forward $(L_1, \epsilon_1)_h$-regular along $E_p^h$, where
$$
L_1 := KL^3/\theta^2
\matsp{and}
\epsilon_1 := (3 + 2\log_\lambda \rho)\epsilon.
$$

Conversely, if $p$ is $N$-times forward $(L, \epsilon)_h$-regular along $E_p^h \in \bbP_p^2$, then there exists $E_p^v \in \bbP_p^2$ such that $p$ is $N$-times forward $(L_2, \epsilon_2)_v$-regular along $E_p^v$, where
$$
L_2 := KL^3
\matsp{and}
\epsilon_2 := (1 + 2\log_\lambda \rho)\epsilon.
$$
\end{prop}

\begin{proof}
Suppose that $p$ is vertically regular along $E_p^v$. For $1 \leq n \leq N$, let $A_n := D_{F^n(p)}F$ and $E^{s_0} := E_p^h$, and define $\chA_n$ as in \subsecref{subsec.proj rep}. We use the same set up as in the proof of \propref{dominated transverse}. Then
$$
\frac{\Jac_p F^n}{\|DF^n|_{E_p^h}\|}=\frac{\Jac A_0^n}{\|A_0^n(v)\|}= \frac{|b_0\ldots b_{n-1}|^2}{\hsigma_n\rho^{(1+\epsilon)n}\|A_0^n(v)\|},
$$
where $v \in E_p^h$ with $\|v\|=1$.

We first establish upper bounds. For $n \leq m$:
\begin{align*}
\frac{\Jac A_0^n}{\|A_0^n(v)\|}&< \frac{K|b_0\ldots b_{n-1}|^2}{\rho^{(1+\epsilon)n}\|v_n\|}\\
&\asymp \frac{K|b_0\ldots b_{n-1}|}{\rho^{(1+\epsilon)n}}\\
&< \frac{K}{|t|}|b_0\ldots b_{n-1}|.
\end{align*}
For $n \geq m$:
\begin{align*}
\frac{\Jac A_0^n}{\|A_0^n(v)\|}&<\frac{K|b_0\ldots b_{n-1}|}{\rho^{(1+\epsilon)m}}\\
&\asymp \frac{K}{|t|}|b_0\ldots b_{n-1}|.
\end{align*}

Next, we establish lower bounds. For $n \leq m$:
\begin{align*}
\frac{\Jac A_0^n}{\|A_0^n(v)\|}&> \frac{|b_0\ldots b_{n-1}|^2}{K\hsigma_n^2\rho^{(1+\epsilon)n}\|v_n\|}\\
&\asymp \frac{|b_0\ldots b_{n-1}|}{K\hsigma_n^2\rho^{(1+\epsilon)n}}\\
&> \frac{1}{KL^2}|b_0\ldots b_{n-1}|
\end{align*}
For $n \geq m$:
\begin{align*}
\frac{\Jac A_0^n}{\|A_0^n(v)\|}&> \frac{|b_0\ldots b_{n-1}|}{K\hsigma_n^2\rho^{(1+\epsilon)m}}\\
&> \frac{1}{KL^2}\rho^{2\epsilon n}|b_0\ldots b_{n-1}|.
\end{align*}

The horizontal projective regularity is given in \propref{dominated transverse}.

Suppose that $p$ is horizontally regular along $E_p^h$. The claimed vertical regularity of $p$ along some direction $E_p^v$ follows immediately from \propref{dominant transverse}.
\end{proof}

\begin{prop}[Horizontal backward regularity $=$ vertical backward regularity]\label{transverse back reg}
There exists a uniform constant $K = K(\rho, \|F\|_{C^1}) \geq 1$ such that the following holds. Suppose $p$ is $M$-times backward $(L, \epsilon)_h$-regular along $E_p^h \in \bbP_p^2$. Let $E_p^v \in \bbP_p^2 \setminus \{E_p^h\}$. If $\measuredangle (E_p^h, E_p^v) > \theta$, then the point $p$ is $M$-times backward $(L_1, \epsilon_1)_v$-regular along $E_p^v$, where
$$
L_1 := KL^3/\theta^2
\matsp{and}
\epsilon_1 := (3 + 2\log_\lambda \rho)\epsilon.
$$

Conversely, if $p$ is $M$-times backward $(L, \epsilon)_v$-regular along $E_p^v \in \bbP_p^2$, then there exists $E_p^h \in \bbP_p^2$ such that $p$ is $M$-times backward $(L_2, \epsilon_2)_h$-regular along $E_p^h$, where
$$
L_2 := KL^3
\matsp{and}
\epsilon_2 := (1 + 2\log_\lambda \rho)\epsilon.
$$
\end{prop}

\begin{prop}[Pesin regularity $=$ vertical forward regularity $+$ horizontal backward regularity $+$ transversality]\label{for back reg}
Suppose $p$ is $N$-times forward $(L, \epsilon)_v$-regular along $E_p^v \in \bbP_p^2$ and $M$-times backward $(L, \epsilon)_h$-regular along $E_p^h \in \bbP_p^2$ with $\theta := \measuredangle (E_p^v, E_p^h) > 0$. Let $\cL \geq 1$ be the minimum value such that $p$ is $(M, N)$-times $(\cL, \bepsilon)_v$-regular along $E_p^v$ and $(\cL, \bepsilon)_h$-regular along $E_p^h$. Then we have $\uL^{-1}\theta^{-2} < \cL < \bL\theta^{-2}$, where
$$
\uL := KL
\comma
\bL := KL^3
\matsp{and}
\bepsilon := (3 + 2\log_\lambda \rho)\epsilon
$$
for some uniform constant $K = K(\rho, \|F\|_{C^1}) \geq 1$. 
\end{prop}

\begin{proof}
The upper bound follows immediately from Propositions \ref{transverse for reg} and \ref{transverse back reg}. The lower bound follows from \propref{dominated transverse}.
\end{proof}

Suppose $p_0 \in \Lambda$ is $(M, N)$-times $(L, \epsilon)_{v/h}$-regular along $E_{p_0}^{v/h} \in \bbP^2_{p_0}$. For $-M \leq m \leq N$, denote
$$
p_m := F^m(p_0)
\matsp{and}
E_{p_m}^{v/h} := D_{p_0}F^m(E_{p_0}^{v/h}).
$$
Define the {\it irregularity} of $p_m$ as the smallest value $\cL_{p_m} \geq 1$ such that $p_m$ is $(M+m, N-m)$-times $(\cL_{p_m}, \epsilon)_{v/h}$-regular along $E_{p_m}^{v/h}$.

\begin{prop}[Growth in irregularity]\label{grow irreg}
Let $\chlambda := \min\{\lambda, \rho\}$. Then
$$
\cL_{p_m} < L^2\chlambda^{-2\epsilon |m|}
\matsp{for}
-M \leq m \leq N.
$$
\end{prop}

\begin{proof}
Denote
$$
b_i := \|D_{p_i}F|_{E_{p_i}^v}\|
\matsp{and}
a_i := \frac{\Jac_{p_i}F}{\|D_{p_i}F|_{E_{p_i}^v}\|}.
$$
For $0 \leq m \leq N$ and $ 0 < n \leq N-m$, we have
$$
L^{-1} \lambda^{(1+\epsilon)(m+n)} \leq b_0 \ldots b_{m+n-1} \leq L \lambda^{(1-\epsilon)(m+n)}.
$$
Since
$$
L^{-1}\lambda^{(1+\epsilon)m} \leq b_0 \ldots b_{m-1} \leq L\lambda^{(1-\epsilon)m},
$$
it follows that
$$
L^{-2}\lambda^{-2\epsilon m} \lambda^{(1+\epsilon)n}<b_m \ldots b_{m+n-1} < L^2\lambda^{-2\epsilon m} \lambda^{(1-\epsilon)n}.
$$
For $0<k \leq m$, we conclude immediately that
$$
L^{-2}\lambda^{2\epsilon k} \lambda^{(1+\epsilon)(m-k)}<b_k \ldots b_{m-1} < L^2\lambda^{-2\epsilon k} \lambda^{(1-\epsilon)(m-k)}.
$$
Lastly, for $0 < l \leq M$, we have
$$
L^{-1}\lambda^{(1+\epsilon)l} \cdot L^{-1}\lambda^{(1+\epsilon)m} \leq b_{-l} \ldots b_{-1} \cdot b_0 \ldots b_{m-1} \leq L\lambda^{(1-\epsilon)l} \cdot L\lambda^{(1-\epsilon)m}.
$$
Similar computations imply analogous bounds for
$$
\frac{b_0 \ldots b_{m+n-1}}{a_0 \ldots a_{m+n-1}}
\comma
\frac{b_m \ldots b_{m+n-1}}{a_m  \ldots a_{m+n-1}}
\matsp{and}
\frac{b_{-l} \ldots b_{-1} \cdot b_0 \ldots b_{m-1}}{a_{-l} \ldots a_{-1} \cdot a_0 \ldots a_{m-1}}.
$$
Thus, the claim holds for $0 \leq m \leq N$.

The proof in the case $-M \leq m < 0$ is nearly identical, and will be omitted.
\end{proof}


\section{Q-Linearization Along Regular Orbits}\label{sec.chart}

In this section, we prove the central result of this paper. We show that any regular orbit can be {\it quasi-linearized} (shortened to ``{\it Q-linearized}''), with controlled error, in a slowly exponentially shrinking neighborhood. The quality of the Q-linearization is explicitly expressed in terms of the regularity parameters. We proceed to construct canonical $C^r$ stable manifolds and canonical $C^r$-germs of central manifolds in the respective regular neighborhoods.

Let $r \geq 1$ be an integer, and consider a $C^{r+1}$-diffeomorphism $F : \Omega \to F(\Omega)\Subset \Omega$ defined on a domain $\Omega \subset \bbR^2$. Suppose a point $p_0\in \Omega$ is $(M, N)$-times $(L, \epsilon, \lambda, \rho)_v$-regular along $E_{p_0}^v \in \bbP^2_{p_0}$ for some $\lambda, \rho, \epsilon \in (0,1)$; $M, N, \in \bbN \cup \{\infty\}$ and $L \geq 1$. We impose the following condition on the contraction bases and marginal exponent:
\begin{equation}\label{eq.reg cond}
\frac{\rho^{(r+1)-\epsilon(r+3)}}{\lambda^{(1+\epsilon)r}} < 1
\matsp{and}
\frac{\lambda^{(1-\epsilon)(r+1)}}{\rho^{r+\epsilon(2-r)}} < 1
\end{equation}

\subsection{Construction of regular charts}

For $l > 0$, denote
$$
\bbB(l) := (-l, l) \times (-l, l) \subset \bbR^2.
$$
For $p \in \bbR^2$, let $E_p^{gv}, E_p^{gh} \in \bbP^2$ be the genuine vertical and horizontal tangent directions at $p$.

\begin{thm}\label{reg chart}
Suppose that \eqref{eq.reg cond} holds. Then there exists a uniform constant
$$
C = C(\|DF\|_{C^r}, \lambda^{-\epsilon}) \geq 1
$$
such that the following holds. For $-M \leq m \leq N$, let
$$
\omega := \frac{\rho^{1-\epsilon}}{1-\rho^{1-\epsilon}}\cdot \|DF^{-1}\|\cdot\|DF\|
\matsp{and}
\cK_{m} := \frac{L^3\rho^{-2\epsilon}\lambda^{1-\epsilon}\|DF^{-1}\|(1+\omega)^5}{\rho^{4\epsilon |m|}\lambda^{2\epsilon |m|}}.
$$
Define
$$
l_{m} :=\chlambda(C\cK_{m})^{-1}
\matsp{and}
U_{m} := \bbB(l_{m}),
$$
where $\chlambda :=  \lambda^{1+\epsilon}(1-\lambda^\epsilon)$. Then there exists a $C^r$-chart $\Phi_{m} : (\cU_{m}, p_m) \to (U_{m}, 0)$ such that $D\Phi_{m}(E^v_{p_m}) = E^{gv}_0$, 
$$
\|D\Phi_{m}^{-1}\|_{C^{r-1}} < C(1+\omega)
\comma
\|D\Phi_{m}\|_{C^s} < C\cK_{m}^{s+1}
\matsp{for}
0 \leq s <r,
$$
and the map $\Phi_{m+1} \circ F|_{\cU_{m}} \circ \Phi_{m}^{-1}$ extends to a globally defined $C^r$-diffeomorphism $F_{m} : (\bbR^2, 0) \to (\bbR^2, 0)$ satisfying the following properties:
\begin{enumerate}[i)]
\item $\displaystyle \|DF_{m}\|_{C^{r-1}} \leq \|DF\|_{C^r}$;
\item we have
$$
D_0F_{m} =\begin{bmatrix}
\alpha_m & 0\\
0 & \beta_m
\end{bmatrix},
$$
where
$$
\alpha_m = \frac{\lambda^{1- \epsilon_m}}{\rho^{1+3\epsilon_m}}
\matsp{and}
\beta_m = \frac{\lambda^{1 - \epsilon_m}}{\rho^{2\epsilon_m}}
\matsp{with}
\epsilon_m := \sgn(m)\cdot \epsilon.
$$
\item $\|D_z F_{m} - D_0F_{m}\|_{C^0} < \chlambda$ for $z \in \bbR^2$;
\item we have
$$
F_{m}(x,y) = (f_{m}(x), e_{m}(x,y))
\matsp{for}
(x,y) \in \bbR^2,
$$
where $f_{m}:(\bbR, 0) \to (\bbR, 0)$ is a $C^r$-diffeomorphism, and $e_{m} : \bbR^2 \to \bbR$ is a $C^r$-map such that for all $0 \leq s \leq r$, we have
$$
\partial_x^s e_{p_m}(\cdot, y) \leq \|DF\|_{C^r} |y|
\matsp{for}
y \in \bbR.
$$
\end{enumerate}
\end{thm}

\begin{proof}
Set
$$
A_m = \begin{bmatrix}
a_m & 0\\
c_m & b_m
\end{bmatrix}
:= D_{p_m}F.
$$
For $n \geq 1$, define
$$
\sigma_n := \frac{\lambda^{(1-\epsilon)n}}{b_0 \ldots b_{n-1}}
\matsp{and}
\sigma_{-n} := \frac{b_{-1} \ldots b_{-n}}{\lambda^{(1+\epsilon)n}}.
$$
For $-M \leq m \leq N$, let
$$
\cS_m := \begin{bmatrix}
\sigma_m & 0 \\
0 & \sigma_m
\end{bmatrix}
\matsp{and}
\tiA_m := \cS_{m+1} \cdot A_m \cdot \cS_m^{-1}
$$
The following properties can be checked by straightforward computations:
\begin{itemize}
\item $\sigma_{n+1}/\sigma_n = \lambda^{1-\epsilon}/b_n$ and $\sigma_{-n+1}/\sigma_{-n} = \lambda^{1+\epsilon}/b_{-n}$;
\item $1 \leq |\sigma_m| \leq L \lambda^{-2\epsilon |m|}$; and
\item we have
$$
\tiA_n =
\begin{bmatrix}
\lambda^{1-\epsilon} a_n/b_n & 0\\
\lambda^{1-\epsilon} c_n/b_n & \lambda^{1-\epsilon}
\end{bmatrix}
\matsp{and}
\tiA_{-n} =
\begin{bmatrix}
\lambda^{1+\epsilon} a_{-n}/b_{-n} & 0\\
\lambda^{1+\epsilon} c_{-n}/b_{-n} & \lambda^{1+\epsilon}
\end{bmatrix}.
$$
\end{itemize}

Define
$$
\hsigma_n := \frac{b_0 \ldots b_{n-1}}{\rho^{(1+\epsilon)n} a_0 \ldots a_{n-1}}
\matsp{and}
\hsigma_{-n} :=\frac{\rho^{(1-\epsilon)n} a_{-n} \ldots a_{-1}}{b_{-n} \ldots b_{-1}}.
$$
For $-M \leq m \leq N$, let
$$
\hcS_m := \begin{bmatrix}
\hsigma_m & 0 \\
0 & 1
\end{bmatrix}
\matsp{and}
\hA_m := \hcS_{m+1} \cdot \tiA_m\cdot \hcS_m^{-1}.
$$
The following properties can be checked by straightforward computations:
\begin{itemize}
\item $\hsigma_{n+1}/\hsigma_n = b_n/(\rho^{1+\epsilon}a_n)$ and $\hsigma_{-n+1}/\hsigma_{-n} = b_{-n}/(\rho^{1-\epsilon}a_{-n})$;
\item $1 \leq |\hsigma_m| \leq L \rho^{-2\epsilon |m|}$;
\item we have
$$
\hA_n =
\begin{bmatrix}
\lambda^{1-\epsilon}/\rho^{1+\epsilon} & 0\\
\hc_n & \lambda^{1-\epsilon}
\end{bmatrix}
\matsp{and}
\hA_{-n} =
\begin{bmatrix}
\lambda^{1+\epsilon}/\rho^{1-\epsilon} & 0\\
\hc_{-n} & \lambda^{1+\epsilon}
\end{bmatrix},
$$
where
\begin{equation}\label{eq.est 1}
\hc_{\pm n} := \frac{\lambda^{1\mp \epsilon} c_{\pm n}}{b_{\pm n} \hsigma_{\pm n}};
\hspace{5mm}
\text{and}
\end{equation}
\item $|\hc_m| < \lambda^{1-\epsilon}\|DF^{-1}\|\cdot\|DF\|$.
\end{itemize}

Define
$$
\bcS_m := \begin{bmatrix}
\rho^{-2\epsilon |m|} & 0\\
0 & \rho^{-2\epsilon |m|}
\end{bmatrix}
\matsp{and}
\bA_m := \bcS_{m+1} \cdot \hA_m \cdot \bcS_m^{-1}.
$$
Then
$$
\bA_n = \begin{bmatrix}
\lambda^{1-\epsilon}/\rho^{1+3\epsilon} & 0\\
\rho^{-2\epsilon} \hc_n & \lambda^{1-\epsilon}\rho^{-2\epsilon}
\end{bmatrix}
\matsp{and}
\bA_{-n} =
\begin{bmatrix}
\lambda^{1+\epsilon}/\rho^{1-3\epsilon} & 0\\
\rho^{2\epsilon} \hc_{-n} & \lambda^{1+\epsilon}\rho^{2\epsilon} 
\end{bmatrix}.
$$

By \propref{proj hor dir}, there exists a genuine horizontal cone $\nabla^{gh}(\omega)$ with
\begin{equation}\label{eq.homega}
\omega := \frac{\rho^{1-\epsilon}}{1-\rho^{1-\epsilon}}\cdot \|DF^{-1}\|\cdot\|DF\|
\end{equation}
such that
$$
\bA_m\left(\nabla^{gh}(\omega)\right)\Subset \nabla^{gh}(\omega).
$$
If $M = \infty$, then by \propref{uni vert dir} (note that the genuine vertical direction $E^{\pi/2}$ is a projective attractor for $\{\bA_{-n}\}_{n=1}^\infty$), there exists a unique direction $\hE_m \in \nabla^{gh}(\omega)$ such that $\bA_m(\hE_m) = \hE_{m+1}$. If $M < \infty$, define
$$
\hE_m := \bA_{m-1} \cdot\ldots \cdot \bA_{-M}(E^0) \in \nabla^{gh}(\omega),
$$
where $E^0$ is the genuine horizontal direction. Define
$$
\cT_m =\begin{bmatrix}
1 & 0\\
-\tau_m & 1
\end{bmatrix}
$$
as the unique matrix such that $\cT_m(\hE_m) = E^0$. Note that $|\tau_m| < \omega$. 

Let
$$
\chA_m := \cT_{m+1}\cdot \bA_m\cdot  \cT_m^{-1}.
$$
Then
$$
\chA_n = \begin{bmatrix}
\lambda^{1-\epsilon}/\rho^{1+3\epsilon} & 0\\
0 & \lambda^{1-\epsilon}\rho^{-2\epsilon}
\end{bmatrix}
\matsp{and}
\chA_{-n} = \begin{bmatrix}
\lambda^{1+\epsilon}/\rho^{1-3\epsilon} & 0\\
0 & \lambda^{1+\epsilon}\rho^{2\epsilon}
\end{bmatrix}.
$$

Let
\begin{equation}\label{eq.zm}
\cZ_m := \kappa \cdot \cT_m\circ \bcS_m\circ \hcS_m \circ \cS_m \circ C_m
\end{equation}
where
\begin{equation}\label{eq.Kzm}
\kappa := L\rho^{-2\epsilon}\lambda^{1-\epsilon}\|DF^{-1}\|(1+\omega)^4,
\end{equation}
and
$C_m$ is the translation and rotation of $\bbR^2$ so that $C_m(E^{\pi/2}_0) = E_{p_m}^v$. Define
$$
\chF_{m} := \cZ_{m+1} \circ F|_{\cU_{m}} \circ \cZ_m^{-1},
$$
where $\cU_{m}$ is a sufficiently small neighborhood of $p_m$ to be specified later. Observe that $D_0 \chF_{m} = \chA_m$. We claim that for any partial derivative $\partial^i$ of order $|i| = s$ with $2 \leq s \leq r+1$, we have $\|\partial^i \chF_{m}\| \leq \|\partial^i F\|$.

First, an elementary computation shows that
\begin{equation}\label{eq.tilt bound}
\max_{|i| = s} \frac{\|\partial^i \chF_{m}\|}{\|\partial^i \left(\cT_{m+1}^{-1} \circ \chF_{m} \circ \cT_m\right)\|} \leq (1+\omega)^{2s}.
\end{equation}
Write
$$
C_{m+1} \circ F \circ C_m^{-1}(x,y) = (f_m(x,y), g_m(x,y)).
$$
Denote
$$
\chf_m :=\pi_h \circ \cT_{m+1}^{-1} \circ \chF_{m} \circ \cT_m
\matsp{and}
\chg_m := \pi_v \circ \cT_{m+1}^{-1} \circ \chF_{m} \circ \cT_m,
$$
where
$$
\pi_h(x,y) := x
\matsp{and}
\pi_v(x,y) :=y.
$$
Then
$$
\chf_m(x,y) = \frac{\kappa\sigma_{m+1}\hsigma_{m+1}}{\rho^{2\epsilon |m+1|}} \cdot f_m\left(\frac{\rho^{2\epsilon |m|}x}{\kappa\sigma_m\hsigma_m},\frac{\rho^{2\epsilon |m|}y}{\kappa\sigma_m}\right)
$$
and
$$
\chg_m(x,y) = \frac{\kappa \sigma_{m+1}}{\rho^{2\epsilon |m+1|}}\cdot g_m\left(\frac{\rho^{2\epsilon |m|}x}{\kappa\sigma_m\hsigma_m},\frac{\rho^{2\epsilon |m|}y}{\kappa\sigma_m}\right).
$$
Taking a partial derivative of order $|i| = s$ with $2 \leq s \leq r+1$, we have
\begin{align*}
\|\partial^i \chf_m\| &\leq \frac{\kappa\sigma_{m+1}\hsigma_{m+1}}{\rho^{2\epsilon|m+1|}} \cdot \left(\frac{\rho^{2\epsilon |m|}}{\kappa\sigma_m}\right)^s\cdot \|\partial^i f_m\|\\
&\leq \frac{1}{\rho^{2\epsilon}\kappa^{s-1}} \cdot \left(\frac{\sigma_{m+1}}{\sigma_m}\right)\cdot\left(\rho^{2\epsilon |m|}\hsigma_{m+1}\right)\cdot \|\partial^i f_m\|\\
&\leq \frac{1}{\rho^{2\epsilon}\kappa^{s-1}} \cdot \lambda^{1-\epsilon}\|DF^{-1}\|\cdot L\cdot \|\partial^i f_m\|\\
&\leq (1+\omega)^{-4(s-1)}\|\partial^i f_m\|.
\end{align*}
By \eqref{eq.tilt bound}, the claimed bound on the higher order partial derivatives of $\chF_{m}$ follows.

Let
$$
\chU_{m} := \bbB(\chlambda \|F\|_{C^2}^{-1}).
$$
Then
\begin{equation}\label{eq.pre c1 dist}
\sup_z\|D_z\chF_{m}|_{\chU_{m}} - \chA_m\| \leq \chlambda.
\end{equation}
Extend $\chF_{m}|_{\chU_{m}}$ to a globally defined $C^{r+1}$-diffeomorphism $\tiF_{m} : (\bbR^2, 0)\to (\bbR^2, 0)$ such that
\begin{itemize}
\item $\tiF_{m}|_{\chU_{m}} \equiv \chF_{m}$; and
\item $\tiF_{m}|_{\bbR^2\setminus \tiU_{m}} \equiv \chA_m$, where $\tiU_{m}$ is a suitable neighborhood of $\chU_{m}$.
\end{itemize}
Additionally, define $\tiF_m$ for $m \in \bbZ \setminus [-M, N]$ by
$$
\tiF_m := \begin{bmatrix}
\lambda/\rho & 0\\
0 & \lambda
\end{bmatrix}.
$$
Then the sequence $\{\tiF_m\}_{m\in\bbZ}$ satisfies the conditions in \appref{subsec.cr sec}.

Thus, the conditions given in \eqref{eq.reg cond}, together with Propositions \ref{for gt} and \ref{back dt} imply that there exist unique sequences of horizontal graphs $\{\Gamma_{g_m^*}\}_{m\in \bbZ}$ and vertical direction fields $\{\cE_{\xi_m^*}\}_{m\in\bbZ}$ that are invariant under $\{\tiF_m\}_{m\in\bbZ}$. Applying \propref{rectify}, we obtain a sequence of $C^r$-charts
$$
\{\Psi_m : (\bbR^2, 0) \to (\bbR^2, 0)\}_{m\in\bbZ}
$$
such that $\|D\Psi_m\|_{C^{r-1}} < C$, and the map
$$
F_{m} := \Psi_{m+1} \circ \tiF_m \circ \Psi_m^{-1}
$$
is of the form given in iii).

Finally,
\begin{equation}\label{eq.phi}
\Phi_{m} := \Psi_m \circ \cZ_m
\end{equation}
gives the desired chart.
\end{proof}

The construction in \thmref{reg chart} is referred to as {\it a Q-linearization of $F$ along the $(M,N)$-orbit of $p_0$ with vertical direction $E^v_{p_0}$}. For $-M \leq m \leq N$, we refer to $l_{m}$, $\cU_{m}$, $\Phi_{m}$ and $F_{m}$ as a {\it regular radius}, a {\it regular neighborhood}, a {\it regular chart} and a {\it Q-linearized map at $p_m$} respectively.

For $p \in \bbR^2$ and $t > 0$, let
$$
\bbD_p(t) := \{\|q - p\| < t\}.
$$

\begin{lem}\label{reg nbh size}
For $-M \leq m \leq N$, we have
$$
\cU_{m} \supset \bbD_{p_m}\left(\frac{\chlambda}{C^2\cK_{m}^2}\right),
$$
where $\chlambda\in(0,1)$ and $C, \cK_{m} \geq 1$ are given in \thmref{reg chart}.
\end{lem}

\begin{proof}
The result follows immediately from the fact that $U_{m} := \bbB(\chlambda C^{-1}\cK_{m}^{-1})$ and $\|D\Phi_{m}\| < C\cK_{m}$.
\end{proof}

For $-M \leq m \leq N$ and $q \in \cU_{m}$, write $z := \Phi_{m}(q)$. The {\it vertical/horizontal direction at $q$ in $\cU_{m}$} is defined as $E^{v/h}_q := D\Phi_{m}^{-1}(E^{gv/gh}_z)$. By the construction of regular charts in \thmref{reg chart}, vertical directions are invariant under $F$ (i.e. $DF(E_q^v) = E_{F(q)}^v$ for $q \in \cU_{m}$). Note that the same is not true for horizontal directions.

\begin{lem}\label{deriv bound}
Define
$$
\chalpha_\pm := \frac{\lambda^{1\mp 2 \epsilon}}{\rho^{1\pm 3\epsilon}}
\matsp{and}
\chbeta_\pm := \frac{\lambda^{1\mp 2 \epsilon}}{\rho^{\pm 2\epsilon}}.
$$
For $-M \leq m \leq N$ and $z \in \bbR^2$, we have
$$
\chalpha_- < |f_{m}'(z)| < \chalpha_+
\matsp{and}
\chbeta_-< |\partial_y e_{m}(z)| < \chbeta_+.
$$
\end{lem}

\begin{proof}
By a straightforward computation, we can check that
$$
\lambda^{1+\epsilon} - \chlambda > \lambda^{1+2\epsilon}
\matsp{and}
\lambda^{1-\epsilon} + \chlambda < \lambda^{1-2\epsilon}.
$$
The result is now an immediate consequence of \thmref{reg chart} ii) and iii).
\end{proof}

For $0 \leq n \leq N-m$, denote
\begin{equation}\label{eq.compose}
F_{m}^n(x,y) = (f_{m}^n(x), e_{m}^n(x,y)) := F_{p_{m+n-1}} \circ \ldots \circ F_{m}(x,y).
\end{equation}

\begin{prop}\label{lin comp}
For $-M \leq m \leq N$ and $0 \leq n \leq N-m$, consider the $C^r$-diffeomorphism $F_{m}^n$ given in \eqref{eq.compose}. Let $z = (x, y) \in U_{m}$, and suppose that
$$
z_i = (x_i, y_i) := F_{m}^i(z) \in U_{{m+i}}
\matsp{for}
0 \leq i \leq n.
$$
Denote
$$
D_zF_{m}^n =: \begin{bmatrix}
\alpha_m^n(z) & 0\\
\gamma_m^n(z) & \beta_m^n(z)
\end{bmatrix}.
$$
Define
$$
\chl_h := \sup_n \frac{2n\chlambda}{C\cK_{n}} < \infty
\matsp{and}
\chl_v := \frac{\chlambda}{C\cK_{0}(1-\rho^{-2\epsilon}\lambda^{1-2\epsilon})}.
$$
and
$$
\kappa_h := \exp\left(\frac{\chl_h\|F\|_{C^2}}{\lambda^{1+2\epsilon}/\rho^{1-3\epsilon}}\right)
\matsp{and}
\kappa_v := \exp\left(\left(\frac{\chl_h}{\lambda^{1+2\epsilon}/\rho^{1-3\epsilon}}+\frac{\chl_v}{\lambda^{1+2\epsilon}\rho^{2\epsilon}}\right)\|F\|_{C^2}\right).
$$
Then
$$
\frac{1}{\kappa_h}\leq \frac{\alpha_m^n(z)}{\alpha_m^n(0)} \leq \kappa_h
\matsp{and}
\frac{1}{\kappa_v}\leq \frac{\beta_m^n(z)}{\beta_m^n(0)} \leq \kappa_v.
$$
Additionally, we have
$$
\|\gamma_m^n\| < \chl_n \chbeta_+^{n-1}\cK_{m}^{-1}
\matsp{where}
\chl_n :=\sum_{i=0}^{n-1}\chalpha_+^i.
$$
\end{prop}

\begin{proof}
Observe that
$$
\sum_{i=0}^{n-1} |x_i| < \chl_h.
$$
We compute
\begin{align*}
\left|\log\left(\frac{(f_{m}^n)'(x_0)}{(f_{m}^n)'(0)}\right)\right| &= \sum_{i=0}^{n-1}\left|\log f_{{m+i}}'(x_i)-\log f_{{m+i}}'(0) \right|\\
&\leq \sum_{i=0}^{n-1}\left(\int_0^{x_i} \left|\frac{f_{{m+i}}''(x)}{f_{{m+i}}'(x)} \right|dx\right)\\
&\leq\frac{\chl_h\|F\|_{C^2}}{\lambda^{1+2\epsilon}/\rho^{1+\epsilon}},
\end{align*}
since $|f_{m}'(x)| > \lambda^{1+2\epsilon}/\rho^{1+\epsilon}$ for all $- M \leq m \leq N$.

Similarly, we have
\begin{align*}
\left|\log\left(\frac{\partial_y e_{m}^n(x_0, 0)}{\partial_y e_{m}^n(0, 0)}\right)\right| &= \sum_{i=0}^{n-1}\left|\log\partial_y e_{{m+i}}(x_i, 0)-\log\partial_y e_{{m+i}}(0, 0) \right|\\
&\leq \sum_{i=0}^{n-1}\left(\int_0^{x_i} \left|\frac{\partial_x\partial_y e_{{m+i}}(x, 0)}{\partial_y e_{{m+i}}(x, 0)} \right|dx\right)\\
&\leq\frac{\chl_h\|F\|_{C^2}}{\lambda^{1+2\epsilon}/\rho^{1-3\epsilon}},
\end{align*}
since $|\partial_y e_{m}(x, y)| > \lambda^{1+2\epsilon}/\rho^{1-3\epsilon}$ for all $- M \leq m \leq N$.

Comparing $\partial_ye_{m}^n(x_0, y_0)$ to $\partial_ye_{m}^n(x_0, 0)$, we first observe that
$$
\sum_{i=0}^{n-1} |y_i| < \chlambda C^{-1}\cK_{0}^{-1}\sum_{i=0}^{n-1}\rho^{-2\epsilon i}\lambda^{(1-2\epsilon)i} < \chl_v.
$$
Thus,
\begin{align*}
\left|\log\left(\frac{\partial_y e_{m}^n(x_0, y_0)}{\partial_y e_{m}^n(x_0, 0)}\right)\right| &= \sum_{i=0}^{n-1}\left|\log\partial_y e_{{m+i}}(x_i, y_i)-\log\partial_y e_{{m+i}}(x_i, 0) \right|\\
&\leq \sum_{i=0}^{n-1}\left(\int_0^{y_i} \left|\frac{\partial_y^2 e_{{m+i}}(x, y)}{\partial_y e_{{m+i}}(x, y)} \right|dy\right)\\
&\leq\frac{\chl_v\|F\|_{C^2}}{\lambda^{1+2\epsilon}\rho^{2\epsilon}}.
\end{align*}

Since $\gamma_m(\cdot, 0) \equiv 0$, we have
$$
|\gamma_m(\cdot,y)| < \|F\|_{C^2}|y|.
$$
Denote $z_n = (x_n, y_n) := F_{m}^n(z)$. By \lemref{deriv bound}, we see that
$$
|y_n| < \chbeta_+^{n-1}|y_0| < \chbeta_+^{n-1}\chlambda C^{-1}\cK_{m}^{-1}.
$$
Arguing by induction, we see that
$$
|\gamma_m^n(z)| < \chbeta_+^{n-1}\left(\sum_{i=0}^{n-1} \chalpha_+^i\right)\chlambda\cK_{m}^{-1}.
$$
\end{proof}

\begin{prop}\label{chart cons}
For $-M \leq m \leq N$ and $q \in \cU_{m}$, we have
$$
\frac{1}{\sqrt{2}} \leq \frac{\|D\Phi_{m}|_{E_z^{v/h}}\|}{\|D\Phi_{m}|_{E_{p_m}^{v/h}}\|} \leq \sqrt{2}.
$$
\end{prop}

\begin{proof}
Recall that $\Phi_{m} := \Psi_m \circ \cZ_m$, and $\diam(U_{m}) < \|\Psi_m\|_{C^2} \cdot \|\cZ_{m}^{\pm 1}\|^{-1}$. Denote
$$
z := \Phi_{m}(q)
\matsp{and}
\tiz := \Psi_m^{-1}(z);
$$
and
$$
\tiE_{\tiz}^{v/h} := D\Psi_m^{-1}(E_z^{gv/gh})
\comma
\tiE_q^{v/h}:= D\cZ_m^{-1}(\tiE_{\tiz}^{v/h})
\matsp{and}
\hE_q^{v/h}:= D\cZ_m^{-1}(\tiE_{\tiz}^{gv/gh}).
$$
We see that
$$
\measuredangle(\tiE_{\tiz}^{v/h}, E_{\tiz}^{gv/gh}) \; , \;
\measuredangle(\tiE_q^{v/h}, \hE_q^{v/h}) < \pi/4.
$$
The result follows immediately.
\end{proof}

\begin{cor}\label{ver hor cons}
For some $-M \leq m_0 \leq N$, let $q_{m_0} \in \cU_{{m_0}}$. Suppose for $m_0 \leq m \leq m_1 \leq N$, we have $q_m \in \cU_{m}$. Let
$$
\hE^h_{q_m} := DF^{m-m_0}(E^h_{q_{m_0}}).
$$
Then for $m_0 \leq m'\leq m_1$, we have
$$
\frac{1}{2\kappa_h}\leq \frac{\|DF^{m'-m}|_{\hE^h_{q_m}}\|}{\|DF^{m'-m}|_{E^h_{p_m}}\|} \leq 2\kappa_h
\matsp{and}
\frac{1}{2\kappa_v}\leq \frac{\|DF^{m'-m}|_{E^v_{q_m}}\|}{\|DF^{m'-m}|_{E^v_{p_m}}\|} \leq 2\kappa_v,
$$
where $\kappa_h$ and $\kappa_v$ are constants given in \propref{lin comp}.
\end{cor}

\begin{proof}
Denote
$$
z_m := \Phi_{m}(q_m)
\matsp{and}
\hE_{z_m}^h := D\Phi_{m}(\hE_{q_m}^h).
$$
By the bound on $\|\gamma_m^n\|$ given in \propref{lin comp}, we see that
$$
\measuredangle(\hE_{z_m}^h, E_{z_m}^{gh}) \; , \;
\measuredangle(\hE_{q_m}^h, E_{q_m}^h) < \pi/4.
$$
The result now follows from Propositions \ref{lin comp} and \ref{chart cons}.
\end{proof}

\begin{lem}\label{tilt bound}
Consider a matrix of the form
$$
T = \begin{bmatrix}
1 & 0\\
\tau & 1
\end{bmatrix}
$$
for some $\tau \in \bbR$. Then $\|T\| < |\tau|+2$.
\end{lem}

\begin{proof}
Let $v = (x,y)$ be a unit vector. Observe
\begin{align*}
\|Tv\|^2 &= x^2+(\tau x + y)^2\\
&\leq 1 + (|\tau|+1)^2\\
&\leq (|\tau|+2)^2. 
\end{align*}
\end{proof}

\begin{prop}[Vertical alignment of forward contracting directions]\label{vert angle shrink}
Consider the constants $C, \omega$ in \thmref{reg chart} and $\kappa_h$ in \propref{lin comp}. Let $q_0 \in \cU_{0}$ and $\tiE_{q_0}^v \in \bbP^2_{q_0}$. Suppose $q_i \in \cU_{ i}$ for $0 \leq i \leq n \leq N$, and that
$$
\nu := \|DF^n|_{\tiE_{q_0}^v}\| < \frac{\lambda^{(1+4\epsilon)n}}{\kappa_h(2+\omega)^3C^2L^2 \rho^{(1-7\epsilon)n}}.
$$
Denote $z_0 := \Phi_{0}(q_0)$ and $\tiE_{z_0}^v := D\Phi_{0}(\tiE^v_{q_0})$. Then
$$
\measuredangle(\tiE^v_{z_0}, E_{z_0}^{gv}) < \left(\frac{\kappa_h(1+\omega)C^2L^2 \rho^{(1-7\epsilon)n}}{\lambda^{(1+4\epsilon)n}}\right)\cdot \nu.
$$
\end{prop}

\begin{proof}
Suppose towards a contradiction that for $u = (1, t) \in \tiE^v_{z_0}$, we have $t \leq 2(1+\omega)$. Using \lemref{deriv bound} and \corref{ver hor cons}, a straightforward computation shows that
$$
\|DF_{0}^n|_{\tiE^v_{z_0}}\| > \frac{\alpha_-^n}{4\kappa_h(\omega+2)}.
$$

Following the construction in \thmref{reg chart}, we see that $\Phi_n := \Psi_m\circ \cZ_m$, where $\cZ_m$ is given in \eqref{eq.zm} and $\|D\Psi_m\|_{C^{r-1}} < C$. In particular, we have $\cZ_0 = \kappa \cT_0 \circ \cC_0$, and
$$
\|D\Phi_m\| < \frac{\kappa CL^2(1+\omega)}{\rho^{4\epsilon n}\lambda^{2\epsilon n}}.
$$
Applying \lemref{tilt bound}, we have
$$
\nu > \|D\Phi_{n}\|^{-1} \cdot\|F_{0}^n|_{\tiE^v_{z_0}}\|\cdot \|D\Phi_0^{-1}\|^{-1} > \frac{\lambda^{(1+4\epsilon)n}}{4\kappa_h(2+\omega)^3C^2L^2 \rho^{(1-7\epsilon)n}}.
$$
Absorbing the constant $4$ into $C$, this is a contradiction.

Suppose that we have $v = (t, 1) \in \tiE^v_{z_0}$ with $t < 1/(2\omega +2)$. Then
$$
\measuredangle(\tiE^v_{q_0}, E^v_{q_0}) \asymp t
\matsp{and}
\|D\cZ_0|_{\tiE^v_{q_0}}\| \asymp 1.
$$
Moreover,
$$
\|DF_{0}^n|_{\tiE^v_{z_0}}\| > \frac{\alpha_-^n t}{2\sqrt{2}\kappa_h}.
$$
Hence, a straightforward computation shows that
$$
\nu >\frac{\lambda^{(1+4\epsilon)n}t}{\kappa_h(1+\omega)C^2L^2 \rho^{(1-7\epsilon)n}}.
$$
\end{proof}

\begin{prop}[Horizontal alignment of backward neutral directions]\label{hor angle shrink}
Consider the constants $C, \omega$ in \thmref{reg chart} and $\kappa_v$ in \propref{lin comp}.  Let $q_0 \in \cU_{0}$ and $\tiE_{q_0}^h \in \bbP^2_{q_0}$. Suppose  $q_{-i} \in \cU_{ {-i}}$ for $0 \leq i \leq n \leq M$, and that
$$
\mu := \|DF^{-n}|_{\tiE^h_{q_0}}\| < \frac{\rho^{6\epsilon n}}{\kappa_v(2+\omega)^3C^2L^2 \lambda^{(1-4\epsilon)n}}.
$$
Denote
$$
z_0 := \Phi_{0}(q_0)
\comma
\tiE_{z_0}^h := D\Phi_{0}(\tiE^h_{q_0})
\matsp{and}
\hE_{z_0}^h := D\Phi_{0} \circ F^n(E^h_{q_{-n}}).
$$
Then
$$
\measuredangle(\tiE_{q_0}^h, \hE_{q_0}^h) <
\left(\frac{\kappa_v(1+\omega)C^2L^2 \lambda^{(1-4\epsilon)n}}{\rho^{6\epsilon n}}\right)\cdot \mu.
$$
\end{prop}

\subsection{Special domains inside regular neighborhoods}

For $w, l > 0$, denote
$$
\bbB(w, l) := (-w, w)\times (-l, l) \subset \bbR^2
\matsp{and}
\bbB(l) := \bbB(l, l).
$$
For $-M \leq m \leq N$, recall that $U_{m} := \bbB(l_m)$, where $l_m$ is the regular radius given in \thmref{reg chart}.

Suppose that
\begin{equation}\label{eq.eps small 1}
\frac{\lambda^{1-4\epsilon}}{\rho^{6\epsilon}} < 1.
\end{equation}
Let
\begin{equation}\label{eq.epm}
e_\pm := \max\{1, \chalpha_\pm^{\pm 1}\},
\end{equation}
where $\chalpha_\pm$ are given in \lemref{deriv bound}. For $0 \leq n_+ \leq N-m$ and $0 \leq n_- \leq M+m$, denote
$$
k_\pm := \max\{|m|, |m \pm n_\pm|\}.
$$
The {\it $n_\pm$-times forward/backward truncated regular neighborhood of $p_0$} is defined as
$$
\cU_{m}^{n,\pm} := \Phi_{m}^{-1}\left(U_{m}^{n,\pm}\right)\subset \cU_{m},
\matsp{where}
U_{0}^{n,\pm} := \bbB\left(e_\pm^{-n} l_{{k_\pm}}, l_{m}\right).
$$

The purpose of truncating a regular neighborhood is to ensure that its iterated images stay inside regular neighborhoods.

\begin{lem}\label{trunc neigh fit}
Let $-M \leq m \leq N$ and $0 \leq n \leq N-m$. We have $F^i(\cU_{m}^{n,+}) \subset \cU_{{m+i}}$ for $0 \leq i \leq n$.
\end{lem}

\begin{proof}
Recall that by \thmref{reg chart} and \lemref{deriv bound}, we have
$$
l_m \asymp \rho^{4\epsilon |m|}\lambda^{2\epsilon |m|}
\matsp{and}
\chbeta_+ := \frac{\lambda^{1- 2 \epsilon}}{\rho^{2\epsilon}}.
$$
Note that $F^i(\cU_{m}^{n,+}) \subset \cU_{{m+i}}$ if $l_m\chbeta_+^i < l_{m+i}$. The result follows from \eqref{eq.eps small 1}.
\end{proof}

For $\omega > 1$ and $a > 0$, define
$$
T_{0}^\omega(a) := \{(x, y) \in U_{0} \; | \; |y| < a|x|^\omega\}.
$$
We refer to
$$
\cT_{0}^\omega(a) := \Phi_{0}^{-1}(T_{0}^\omega(a)) \subset \cU_{0}
$$
as a {\it $\omega$-pinched regular neighborhood (of dilation $a$) at $p_0$}. Note that $p_0 \not\in \cT_{0}^\omega(a)$. For $n \geq 0$, we also denote
$$
T_{0}^{\omega, n}(a) := T_{0}^\omega(a) \cap U_{0}^{n, -}
$$
and
$$
\cT_{0}^{\omega, n}(a):= \Phi_{0}^{-1}(T_{0}^{\omega, n}(a))=\cT_{0}^\omega(a) \cap \cU_{0}^{n, -}.
$$
If the dilation $a$ is not indicated, it is assumed to be equal to $1$: e.g. $\cT_{0}^\omega := \cT_{0}^\omega(1)$.

The purpose of pinching a regular neighborhood is to ensure that its iterated preimages stay inside regular neighborhoods.

\begin{lem}\label{pinch neigh fit}
Let $0 \leq n \leq M$. If
$$
\omega > \frac{\log\left(\lambda^{1+4\epsilon}\rho^{6\epsilon}\right)}{\log\left(\lambda^{1+4\epsilon}/\rho^{1-7\epsilon}\right)},
$$
then $F^{-i}(\cT_{0}^{\omega, n}) \subset \cU_{{-i}}$ for $0 \leq i \leq n$.
\end{lem}

\begin{proof}
By Lemmas \ref{reg nbh size} and \ref{deriv bound}, we see that
$$
F_{{-n}}^n(U_{{-n}}^{n,+}) \supset \bbB(l_{{-n}}e_-^{-n}, l_{{-n}}\beta_+^n) \supset T_{0}^{\omega, n, -}.
$$
\end{proof}

\begin{lem}\label{pinching neigh}
For $q_0 \in \cU_{0} \setminus \{p_0\}$, let $n \in \bbN$ be the largest number such that $q_0 \in \cU_{0}^{n,-}$. If $q_{-n} \in \cU_{{-n}}$, then $q_0 \in \cT_{0}^{\omega, n}(C)$, where
$$
\omega = \max\left\{\frac{\log(\lambda^{1-2\epsilon}/\rho^{2\epsilon})}{\log(\lambda^{2\epsilon}\rho^{4\epsilon})}, \frac{\log(\lambda\rho^{2\epsilon})}{\log(\lambda^{1+4\epsilon}/\rho^{1-7\epsilon})}\right\}
$$
and
$$
C = \max\left\{\frac{1}{\lambda^{2\epsilon\omega}\rho^{4\epsilon \omega}l_{0}^{\omega-1}},\frac{\rho^{(1-7\epsilon)\omega}}{\lambda^{(1+4\epsilon)\omega}l_{0}^{\omega-1}} \right\}.
$$
\end{lem}

\begin{proof}
Denote $z_0 = (x_0, y_0) := \Phi_{0}(q_0)$. We have
$$
e_-^{-(n+1)} l_{{n+1}} < |x_0| < e_-^{-n}  l_{n}
$$
Moreover, by \lemref{deriv bound}, we see that $|y_0| < l_{n}\chbeta_+^n$. A straightforward computation shows that $|y_0| < C|x_0|^\omega$.
\end{proof}

\subsection{Stable manifolds}

For $-M \leq m \leq N$, define the {\it local vertical} and {\it horizontal manifold at $p_m$} as
$$
W^v_{\loc}(p_m) := \Phi_{m}^{-1}(\{(0, y) \in U_{m}\})
\matsp{and}
W^h_{\loc}(p_m) := \Phi_{m}^{-1}(\{(x,0) \in U_{m}\})
$$
respectively.

If $N = \infty$, then \propref{dominated transverse} implies that $E_{p_0}^v$ is the unique direction along which $p_0$ is infinitely forward regular (see \remref{rem.reg is proj reg}). In this case, we denote $E_{p_0}^{ss} := E_{p_0}^v$, and refer to this direction as the {\it strong stable direction at $p_0$}. Additionally, we define the {\it strong stable manifold of $p_0$} as
$$
W^{ss}(p_0) := \left\{q_0 \in \Omega \; | \; \limsup_{n \to \infty}\frac{1}{n}\log\|q_n - p_n\| < (1-\epsilon)\log\lambda\right\}.
$$

\begin{thm}[Canonical strong stable manifold]\label{stable}
Suppose
\begin{equation}\label{eq.stable cond}
\lambda^{1-\epsilon} < \rho^{8\epsilon}\lambda^{4\epsilon}
\matsp{and}
\lambda^{1-\epsilon} < \frac{\lambda^{1+\epsilon}}{\rho^{1-\epsilon}}.
\end{equation}
If $N = \infty$, then
$$
W^{ss}(p_0) := \bigcup_{n=0}^\infty F^{-n}(W^v_{\loc}(p_n)).
$$
Consequently, $W^{ss}(p_0)$ is a $C^{r+1}$-smooth manifold.
\end{thm}

\begin{proof}
Clearly, $W^v_{\loc}(p_0) \subset W^{ss}(p_0)$. Let $q_0 \in W^{ss}(p_0)$. By the first inequality in \eqref{eq.stable cond} and \lemref{reg nbh size}, we see that $q_n \in \cU_{n}$ for all $n$ sufficiently large. Denote $z_n := \Phi_{n}(q_n)$. We see by the first inequality in \eqref{eq.stable cond} and \lemref{pinching neigh} that either $q_n \in W^v_{\loc}(p_n)$, or $q_n \in \cT^\omega_{n}(l_{n}, C)$, where $\omega > 1$, and $C > 0$ is a uniform constant. Thus, $z_n$ converges to $0$ tangentially along the horizontal direction. However, by \thmref{reg chart} and \propref{chart cons}, we have
$$
\liminf_{n\to \infty}\frac{1}{n}\log\|q_n - p_n\| > \log\left(\frac{\lambda^{1-\epsilon}}{\rho^{1+3\epsilon}} \cdot\rho^{4\epsilon}\lambda^{2\epsilon}\right) = \log\left(\frac{\lambda^{1+\epsilon}}{\rho^{1-\epsilon}}\right).
$$
This contradicts the second inequality in \eqref{eq.stable cond}.
\end{proof}

\subsection{Neutral direction}

If $\rho = \lambda$, then instead of ``$(L, \epsilon, \lambda, \lambda)_{v/h}$-regular'', we simply write ``$(L, \epsilon, \lambda)_{v/h}$-regular.'' Suppose that $p_0$ is $(M, N)$-times $(L, \epsilon, \lambda)_v$-regular along $E_{p_0}^v$.

\begin{prop}[Jacobian bounds]\label{jac bound}
We have
$$
L^{-3}\lambda^{(1+3\epsilon)n} \leq \Jac_{p_0}F^n \leq L^3\lambda^{(1-3\epsilon)n}
\matsp{for}
1 \leq n \leq N,
$$
and
$$
L^{-3}\lambda^{-(1-3\epsilon)n}\leq \Jac_{p_0}F^{-n} \leq L^3\lambda^{-(1+3\epsilon)n}
\matsp{for}
1 \leq n \leq M.
$$
\end{prop}

\begin{prop}[Derivative bounds]\label{ext deriv bound}
Let $C \geq 1$ and $\omega > 0$ be the constants given in \thmref{reg chart}. Then for $E_{p_0} \in \bbP^2_{p_0}$, we have
$$
\frac{\lambda^{(1+3\epsilon)n}}{CL^2(1+\omega)^2} \leq \|DF^n|_{E_{p_0}}\| \leq C(1+\omega)^2\lambda^{-4\epsilon n}
\matsp{for}
1 \leq n \leq N,
$$
and
$$
\frac{\lambda^{2\epsilon n}}{CL^2(1+\omega)^2}\leq \|DF^{-n}|_{E_{p_0}}\| \leq C(1+\omega)^2\lambda^{-(1+3\epsilon)n}
\matsp{for}
1 \leq n \leq M.
$$
\end{prop}

\begin{proof}
We prove the inequalities for the forward direction. The proof of the backward direction is similar.

Decompose $F^n = \Phi_{n}^{-1} \circ F_{0}^n\circ \Phi_{0}.$ Let
$$
\kappa := L\rho^{-2\epsilon}\lambda^{1-\epsilon}\|DF^{-1}\|(1+\omega)^4.
$$
By \eqref{eq.zm}, \eqref{eq.phi} and \thmref{reg chart} ii), we see that $\|DF_{0}^n\| < \lambda^{-4\epsilon n}$,
$$
\|D\Phi_{0}\| < C\kappa
\matsp{and}
\|D\Phi_{n}^{-1}\| < \kappa^{-1}C(1+\omega).
$$
After replacing $C$ by $C^2$, the upper bound follows. For the lower bound, we similarly observe that $\|DF_{0}^n|_{E_{\Phi_0(p_0)}}\| > \lambda^{(1-3\epsilon)n}$,
$$
\|D\Phi_{0}|_{E_{p_0}}\| > \frac{\kappa}{C(1+\omega)}
\matsp{and}
\|D\Phi_{n}|_{E_{p_n}}\| <  \kappa (1+\omega)CL^2\lambda^{-6\epsilon n}
$$
By uniformly increasing $C$ if necessary, the lower bound also follows.
\end{proof}

If $M = \infty$, then \propref{dominated transverse} implies that $E_{p_0}^h$ is the unique direction along which $p_0$ is infinitely backward regular (see \remref{rem.reg is proj reg}). In this case, we denote $E_{p_0}^c := E_{p_0}^h$, and refer to this direction as the {\it center direction at $p_0$}. Moreover, we define the {\it (local) center manifold at $p_0$} as
$$
W^c(p_0) := \Phi_{0}^{-1}(\{(x, 0) \in U_{0}\}).
$$
Unlike strong stable manifolds, center manifolds are not canonically defined. However, the following result states that it still has a canonical jet.

\begin{thm}[Canonical jets of center manifolds]\label{center jet}
Suppose
\begin{equation}\label{eq.center cond}
\epsilon < (11r+2)^{-1}.
\end{equation}
Assume $M = \infty$. Let $\Gamma_0: (-t, t) \to \cU_{0}$ be a $C^{r+1}$-curve parameterized by its arclength such that $\Gamma_0(0) = p_0$, and for all $n \in \bbN$, we have
$$
\|DF^{-n}|_{\Gamma_0'(t)}\| < \lambda^{-\mathfrak{r} n}
\matsp{for}
|t| < \lambda^{11\epsilon n},
$$
where
$$
\mathfrak{r} := 1-\epsilon(1-\epsilon)^{-1}(7+11r+2\epsilon + 66r\epsilon).
$$
Then $\Gamma_0$ has a degree $r+1$ tangency with $W^c(p_0)$ at $p_0$.
\end{thm}

\begin{proof}
Let $\cN_0$ be a sufficiently small neighborhood of $p_0$, and denote $\gamma_0 := \Phi_{0}(\Gamma_0 \cap \cN_{0})$. If $\Gamma_0$ does not have a degree $r+1$ tangency with $W^c(p_0)$ at $p_0$, then there exists $a_0 \neq 0$ and $l_0 \in (0,1)$ such that
$$
\gamma_0 = \{(x,a_0 x^{r+1}) \; | \; |x| < l_0\}.
$$

Let $z_0 = (x_0, y_0) \in \gamma_0 \cap U_{0}^{n, -}$. This means
$$
|x_0| < e_-^nl_{{-n}}=\lambda^{5\epsilon n}\lambda^{6\epsilon n} =\lambda^{11\epsilon n} .
$$
Denote $z_{-n} = (x_{-n}, y_{-n}) := (F_{{-n}}^n)^{-1}(z_0)$. By \thmref{reg chart} ii), we see that
$$
y_{-n} < \lambda^{-(1+3\epsilon)n} y_0.
$$
Since $z_{-n} \in U_{{-n}}$ if $y_{-n} < l_{{-n}}$, we conclude that $z_{-n} \in U_{{-n}}$ if
$$
y_0 <  \lambda^{(1+3\epsilon)n}\lambda^{6\epsilon n} = \lambda^{(1+9\epsilon)n}.
$$
This can be verified by a straightforward computation using \eqref{eq.center cond}.

Let $E_{z_0} \in \bbP^2_{z_0}$ be the tangent direction to $\gamma_0$ at $z_0$. Denote
$$
E_{z_{-n}} := D(F_{{-n}}^n)^{-1}(E_{z_0})
\matsp{and}
t_{-n} := \measuredangle(E_{z_{-n}}, E_{z_{-n}}^{gh}).
$$
Note that $t_0 = (r+1)a_0 x_0^r$. Again, by \thmref{reg chart} ii), we have $t_{-m} > \lambda^{-(1-\epsilon)m}t_0$. Set $x_0 = \lambda^{11\epsilon n}$, and let $0 \leq n_0 \leq n$ be the smallest number such that $t_{-n_0} > 1$. By a straightforward computation, we conclude that
$$
n_0 \asymp \frac{11\epsilon r}{1-\epsilon} n.
$$

Denote $q_0 := \Phi_{0}^{-1}(z_0)$ and $E_{q_0} := D\Phi_{0}^{-1}(E_{z_0})$. Using \thmref{reg chart}, \lemref{deriv bound} and \propref{chart cons}, we obtain
\begin{align*}
\|DF^{-n}|_{E_{q_0}}\| &> \|D\Phi_{-n}|_{E_{q_{-n}}}\|^{-1} \cdot \|D(F_{{-n+n_0}}^{n-n_0})^{-1}|_{E_{z_{n_0}}}\| \cdot \|D(F_{{-n_0}}^{n_0})^{-1}|_{E_{z_0}}\|\\
&> \lambda^{6\epsilon n} \cdot \lambda^{-(1+2\epsilon)(n-n_0)} \cdot \lambda^{4\epsilon n_0}\\
&>\lambda^{-\mathfrak{r} n}.
\end{align*}
\end{proof}


\section{Homogeneity}\label{sec:homog}

In this section, we introduce the notion of homogeneity that amounts to little variation of the Jacobian over the phase space. We show that uniquely ergodic non-uniformly partially hyperbolic maps can be made homogeneous after passing to some iterate. In \cite{CLPY1}, homogeneity is used to build further refinements of quantitative Pesin theory presented in this paper.

Let $r \geq 2$ be an integer, and consider a $C^{r+1}$-diffeomorphism $F: \Omega \to F(\Omega)\Subset \Omega$ defined on a domain $\Omega \Subset \bbR^2$. Suppose that $\Lambda \Subset \Omega$ is a compact totally invariant set for $F$ not equal to the orbit of a periodic sink. For $\lambda, \eta \in (0,1)$, we say that $F$ is {\it $(\eta, \lambda)$-homogeneous on $\Lambda$} if for all $p \in \Lambda$ and $E_p \in \bbP_p^2$, we have:
\begin{enumerate}[i)]
\item $ \lambda^{1+\eta} < \|D_pF|_{E_p}\| < \lambda^{-\eta}$ and
\item $\lambda^{1+\eta} < \Jac_p F < \lambda^{1-\eta}$.
\end{enumerate}

\begin{prop}\label{get homog}
Suppose $F|_\Lambda$ has a unique invariant probability measure $\mu$, and that with respect to $\mu$, the Lyapunov exponents of $F|_\Lambda$ are $\log\lambda <0$ and $0$. Then for any $\eta>0$, there exists $N = N(\eta) \in \bbN$ such that if $n \geq N$, then the map $F^n$ is $(\eta, \lambda)$-homogeneous on $\Lambda$.
\end{prop}

\begin{proof}
For all $p \in \Lambda$, we have
$$
\frac{1}{n}\log\left(\Jac_pF^n\right) = \frac{1}{n}\sum_{i =0}^{n-1}\log\left(\Jac_{F^i(p)}F\right) \xrightarrow{n \to \infty} \int \log \left(\Jac F\right) d\mu = \log\lambda.
$$
This implies that for $\eta_1 >0$, there exists $N \in \bbN$ such that for all $n > N$, we have
$$
\left|\frac{1}{n}\log\left(\Jac_pF^n\right) - \log\lambda\right| < \eta_1.
$$
Thus,
$$
(\lambda^n)^{1+\eta_1/\log\lambda} < \Jac_pF^n < (\lambda^n)^{1-\eta_1/\log\lambda}.
$$

Since the largest Lyapunov exponent of $F|_\Lambda$ is $0$, for $\eta_2 >0$, there exists $M \in \bbN$ such that for all $m > M$, we have
$$
\frac{1}{m}\int\log\|D F^m\|d\mu < \eta_2.
$$
For all $p \in \Lambda$, we have
$$
\frac{1}{n}\log\|D_p F^n\| \leq \frac{1}{n}\sum_{i =0}^{n-1} \left(\frac{1}{m}\log\|D_{F^i(p)} F^m\|\right) \xrightarrow{n \to \infty} \frac{1}{m}\int\log\|D F^m\|d\mu.
$$
Thus, there exists $N' \in \bbN$ such that for all $n > N'$, we have
$$
\|D_p F^n\| \leq e^{\eta_2 n} = (\lambda^n)^{\eta_2/\log\lambda}.
$$

Lastly, let
$$
k_n := \min_{E_p \in \bbP^2_p}\|DF^n|_{E_p}\|.
$$
Then
$$
\Jac_pF^n = k_n\|D_pF^n\|.
$$
The result follows.
\end{proof}

Suppose that $F$ is $(\eta, \lambda)$-homogeneous on $\Lambda$. Let $C, D > 1$ be sufficiently large uniform constants independent of $F$. Let $0 < \epsilon < \eta < 1$ be sufficiently small constants such that
$$
\baeta := C\eta^{1/D} < \epsilon
\matsp{and}
\bepsilon :=C\epsilon^{1/D} < 1.
$$
Homogeneity considerably simplifies the regularity conditions given in \secref{sec.defn reg}. Let $N, M \in \bbN\cup \{0,\infty\}$ and $L \geq 1$. Then a point $p \in \Lambda$ is:
\begin{itemize}
\item $N$-times forward $(L, \epsilon, \lambda, \lambda)_v$-regular along $E^v_p \in \bbP^2_p$ if
$$
\|DF^n|_{E^v_p}\| \leq L \lambda^{(1-\epsilon)n}
\matsp{for}
1 \leq n \leq N;
$$
\item $M$-times backward $(L, \epsilon, \lambda, \lambda)_v$-regular along $E^v_p$ if
$$
\|DF^{-n}|_{E^v_p}\| \geq L^{-1} \lambda^{-(1-\epsilon) n}
\matsp{for}
1 \leq n \leq M;
$$
\item $N$-times forward $(L, \epsilon, \lambda, \lambda)_h$-regular along $E^h_p \in \bbP^2_p$ if
$$
\|DF^n|_{E^h_p}\| \geq L^{-1} \lambda^{\epsilon n}
\matsp{for}
1 \leq n \leq N;
$$
\item $M$-times backward $(L, \epsilon, \lambda, \lambda)_h$-regular along $E^h_p \in \bbP^2_p$ if
$$
\|DF^{-n}|_{E^h_p}\| \leq L \lambda^{-\epsilon n}
\matsp{for}
1 \leq n \leq M.
$$
\end{itemize}

\appendix

\section{Adaptation Lemma}

Fix $\lambda, \epsilon \in (0, 1)$. Let $\{u_m\}_{m=-M}^{N-1}$ be a sequence of numbers with $M, N \in \bbN \cup \{0, \infty\}$. Moreover, suppose there exists $L \geq 1$ such that
$$
L^{-1} \lambda^{(1+\epsilon) n} \leq u_0\ldots u_{n-1} \leq L\lambda^{(1-\epsilon)n}
\matsp{for}
1 \leq n \leq N;
$$
and
$$
L^{-1} \lambda^{(1+\epsilon) m} \leq u_{-1}\ldots u_{-m} \leq L\lambda^{(1-\epsilon)m}
\matsp{for}
1 \leq m \leq M.
$$
Choose $\delta \in (\epsilon, 1)$. Define
$$
\zeta_n := \frac{\lambda^{(1-\delta)n}}{u_0 \ldots u_{n-1}}
\matsp{and}
\zeta_{-m} := \frac{u_{-1} \ldots u_{-m}}{\lambda^{(1+\delta)m}}
$$

We record the following useful properties, which can be checked by straightforward computations.

\begin{lem}\label{adapt}
For $0 \leq n < N$ and $1 \leq m \leq M$, we have
$$
\frac{\zeta_{n+1}}{\zeta_n} = \frac{\lambda^{1-\delta}}{u_n}
\matsp{and}
\frac{\zeta_{-m+1}}{\zeta_{-m}} = \frac{\lambda^{1+\delta}}{u_{-m}}.
$$
Moreover,
$$
1 < \zeta_l < L\lambda^{-(\delta+\epsilon)|l|}
\matsp{for}
-M \leq l < N.
$$
\end{lem}


\section{$C^r$-bounds under Compositions}

\begin{lem}\cite[(4)]{PuSh}\label{hocb}
Let $F, G$ be $C^r$-maps such that $F \circ G$ is well-defined. Then
$$
\|F\circ G\|_r \leq r^r \|F\|_r \|G\|_r^r,
$$
where $\|F\|_r := \|DF\|_{C^{r-1}}$.
\end{lem}

\begin{lem}[Product]\label{cr product}
Let $d \in \bbN$. Consider $C^r$-maps $F, G : U \to \bbR$ defined on $U \subset \bbR^d$ such that $\|DF\|_{C^{r-1}}, \|DG\|_{C^{r-1}} < \infty$. Then there exists a uniform constant $C = C(r) \geq 1$ such that
$$
\|D(F\cdot G)\|_{C^{r-1}} < C\|DF\|_{C^{r-1}}\|DG\|_{C^{r-1}}.
$$
\end{lem}

\begin{lem}[Quotient]\label{cr quotient}
Let $d \in \bbN$. Consider a $C^r$-map $F : U \to \bbR$ defined on $U \subset \bbR^d$ such that $|F(x)| > \mu > 0$ for $x\in U$, and $\|DF\|_{C^{r-1}} < \infty$. Then there exists a uniform constant $C = C(r) \geq 1$ such that
$$
\|D(1/F)\|_{C^{r-1}} < C\mu^{1-r}\|DF\|_{C^{r-1}}^{r-1}.
$$
\end{lem}

\begin{proof}
The proof is by induction on $r$. Suppose that an $r$th order partial derivative $\partial^r (1/F)(x)$ is a sum of uniform number of terms of the form $P\cdot (F(x))^{-l}$, where $l < r$, and $P$ is a degree $k < r$ polynomial expression of partial derivatives of $F$ up to order $r$.

Differentiating, we obtain
$$
\frac{\partial P}{(F(x))^l} - \frac{P \cdot l\partial F(x)}{(F(x))^{l+1}}.
$$
The result follows.
\end{proof}

\begin{lem}[Inverse]\label{cr inverse}
Consider a $C^r$-diffeomorphism $f : \bbR \to \bbR$. Suppose $\|f'\| > \mu$ for some constant $\mu \in (0,1)$. Then there exists a uniform constant $C = C(r) \geq 1$ such that
$$
\|(f^{-1})'\|_{C^{r-1}} < C\mu^{1-2r}\|f''\|_{C^{r-2}}^{r-1}.
$$
\end{lem}

\begin{proof}
The proof is by induction on $r$. Let $u := f^{-1}(x)$ for $x \in \bbR$. Note that $u' = 1/(f'(u))$. The case $r = 1$ follows. Suppose that $(f^{-1})^{(r)}(x)$ is a sum of uniform number of terms of the form $P\cdot (f'(u))^{-l}$, where $l < 2r$, and $P$ is a degree $k < r$ polynomial expression of $f^{(i)}(u)$ for $2 \leq i \leq r$.

Differentiating, we obtain 
$$
\frac{P' \cdot u'}{(f'(u))^l}-\frac{P\cdot lf''(u) \cdot u'}{(f'(u))^{l+1}},
$$
The result follows.
\end{proof}

\begin{prop}[Compositions of 1D Diffeomorphisms]\label{cr hor}
Consider a sequence $f_n : \bbR \to \bbR$ for $n \geq 0$ of $C^r$-diffeomorphisms. Denote $f^n := f_{n-1} \circ \ldots \circ f_0$. Suppose $|f_n'| < \mu$ and $\|f_n'\|_{C^{r-1}} < \bfC$ for some constants $\mu, \bfC > 1$. Then there exists a uniform constant $C = C(\bfC, r) \geq 1$ such that
\begin{equation}\label{e.f}
\|Df^n\|_{C^{r-1}} < C\mu^{r(n-1)}.
\end{equation}
\end{prop}

\begin{proof}
To simplify notation used in the proof, we denote $f_m$ as simply $f$.

Let us set $C_1(1,\bfC)=\bfC\geq \mu$ and define inductively:
$$(1-\tfrac 1 \mu) C_1(k,\bfC)=2^kK_k\bfC C_1(k-1,\bfC)^{k},$$
where $K_k\geq 1$ is an integer which only depends on $k$ and is defined below.

Note $C_1(k,\bfC)\geq C_1(k-1,\bfC)$.
Since $\|Df\|_r\leq \bfC\leq C_1(r,\bfC)$, \eqref{e.f} holds for $n=1$.
Since $\|Df^n\|\leq \mu^n\leq C_1(1,\bfC)\mu^{n-1}$, it also holds for $r=1$.

For larger values of $n$ and $r$ this is proved inductively on $n$.
Assuming that \eqref{e.f} holds for $n$ and for all $k\leq r$, the Fa\`a di Bruno formula gives:
$$
D^k f^{n+1} =
\sum_{\substack{k=a_1+\cdots +a_\alpha\\
1\leq a_i< k}}
\bigg[ K(a_1,\dots,a_\alpha)
D^\alpha f_n \prod_{i=1}^{\alpha} D^{a_i} f^n \bigg]
+Df_n\cdot D^kf^n,
$$
where the sum is over all decompositions $k=a_1+\cdots +a_\alpha$
with $1\leq a_i<k$ and where the integer $K(a_1,\dots,a_\alpha)$ only depends on such a decomposition.
The number of such decompositions is bounded by $2^k$.

Let $K_k$ be the maximum of all the $K(a_1,\dots,a_\alpha)$
with $a_1+\cdots +a_\alpha\leq k$. We get:
$$\|D^kf^{n+1}\|
\leq \sum_{\substack{k=a_1+\cdots +a_\alpha\\
1\leq a_i<k}} K_{k} \bfC C_1(k-1,\bfC)^\alpha \mu^{(a_1+\dots+a_\alpha)(n-1)}
+ \mu\cdot C_1(k,\bfC)\cdot \mu^{k(n-1)}$$
$$\leq [2^k K_{k} \bfC C_1(k-1,\bfC)^{k} + \tfrac 1 \mu C_1(k,\bfC)]\mu^{kn} \leq C_1(k,\bfC)\mu^{kn}.$$
Taking the supremum over all values of $k\leq r$, this proves \eqref{e.f} for $n+1$.
\end{proof}

\begin{prop}[Compositions of 2D contractions]\label{sqze}
Consider a sequence $F_n : \bbR \times (-1, 1) \to \bbR \times (-1, 1)$ for $n \geq 0$ of $C^r$-diffeomorphisms of the form
$$
F_n(x,y) = (f_n(x), e_n(x,y))
\matsp{for}
(x,y)\in\bbR \times (-1, 1)
$$
where $f_n : \bbR \to \bbR$ is a $C^r$-diffeomorphism, and $e_n : \bbR\times (-1, 1) \to \bbR$ is a $C^r$-map. For $n \in \bbN$, denote
$$
F^n = (f^n, e^n) := F_{n-1} \circ \ldots \circ F_0.
$$
Suppose there exist constants $\mu, \bfC > 1$ and $\lambda \in (0,1)$  such that $\|DF_m\|_{r-1}<\bfC$, $|f'_m|<\mu$, $|\tfrac {\partial e}{\partial y}|\leq \lambda $
and for $0\leq k\leq r$ and for all $e_n$ holds
$$
|\tfrac {\partial^k}{\partial x^k} e_n(x,y)|\leq \bfC|y| 
\matsp{for}
y \in (-1, 1).
$$ 
Then for each $r$ there exists a constant $C(r, \bfC) > 0$ such that
\begin{equation}\label{e.e}
\|De^n\|_{C^{r-1}} \leq C\mu^{r(n-1)}\lambda^{n-1}.
\end{equation}
In particular, if $\mu^r\lambda<1$ it follows that the $C^r$-norm of $e^n$ converges exponentially to zero.
\end{prop}

\begin{proof}
To simplify notation used in the proof, we denote $e_m$ as simply $e$.

The Fa\`a di Bruno formula for higher derivatives gives:

\begin{equation}\label{fdb}
 \frac{\partial^r e^{n+1}}{\partial x^{r_1} \partial y^{r_2}} = \sum_{\substack{r_1=a_1+\cdots +a_\alpha\\
\quad +b_1+\dots+b_\beta\\
r_2=c_1+\dots+c_\beta}}
K((a_i),(b_j),(c_j))
\bigg[\frac{\partial^{\alpha+\beta} e_n}{\partial x^{\alpha} \partial y^{\beta}}(f^n,e^n)\prod_{i=1}^{\alpha} \frac{\partial^{a_i} f^n}{\partial x^{a_{i}}} \prod_{j=1}^{\beta}\frac{\partial^{b_j+c_j} e^n}{\partial x^{b_j} \partial y^{c_j}} \bigg] 
\end{equation}
\noindent
where the sum is over all decompositions $r_1=a_1+\cdots +a_\alpha+b_1+\dots+b_\beta$ and $r_2=c_1+\dots+c_\beta$
for all possible values of $\alpha,\beta\geq 0$ with $a_i\geq 1$, $b_j+c_j\geq 1$
and where $K((a_i),(b_j),(c_j))$ is an integer which only depends on such a decomposition. Let us introduce $A_\alpha=\sum_{i=1}^\alpha a_i,$ $B_\beta=\sum_{j=1}^\beta b_j+c_j$ so that $A_\alpha+B_\alpha=r$.

\vskip 10pt

\paragraph{\bf{I- Bounds on $\frac{\partial^{\alpha} e}{\partial x^\alpha}\circ(f^n,e^n)$  (assuming $\beta=0$)}}  Using that $|\tfrac {\partial e}{\partial y}|\leq \lambda$ we get $|e^n(x, y)| <\lambda^n$. Since $|\frac{\partial^{\alpha} e}{\partial x^\alpha}(x,y)|< \bfC|y|$ then we get 
$$\bigg\|\frac{\partial^{\alpha} e}{\partial x^{\alpha}}\circ(f^n,e^n)\bigg\|< \bfC. \lambda^n.$$

\vskip 10pt

\paragraph{\bf{II-  Bounds on $\prod_{i=1}^{\alpha} \partial^{a_i} f^n$}} By proposition \ref{cr hor} we get

$$\bigg\|\prod_{i=1}^{\alpha} \partial^{a_i} f^n \bigg\|<   C_1(r,\bfC)^\alpha\mu^{n.A_\alpha}.$$


\vskip 10pt

\paragraph{\bf{III- Bounds on $\prod_{j=1}^{\beta} \frac{\partial^{b_{j} + c_{j}} e^n}{\partial x^{b_{j}} \partial y^{c_{j}}}$}}
Let $k=\sup\{b_j+c_j\}$. The inductive hypothesis gives

$$\bigg\|\prod_{j=1}^{\beta} \frac{\partial^{b_{j} + c_{j}} e^n}{\partial x^{b_{j}} \partial y^{c_{j}}}\bigg\| <  \prod_{j=1}^{\beta} C_2(b_j+c_j)
\mu^{(b_j+c_j)(n-1)}\lambda^{n-1}
\leq C_2(k,\bfC)^\beta \mu^{B_\alpha.(n-1)}  \lambda^{\beta(n-1)}.$$

\vskip 10pt

\paragraph{\bf{Concluding}} We analyze each term of the sum by looking to different values of $\alpha, \beta$.

\vskip 5pt

\noindent {\it Case $\beta=0$ and $\alpha>0$.} 
From (I) and (II) it follows  that  
\begin{eqnarray*}
\bigg\|\frac{\partial^{\alpha+\beta} e_{n+1}}{\partial x^{\alpha} \partial y^{\beta}}(f^n,e^n)\prod_{i=1}^{\alpha} \frac{\partial^{a_i} f^n}{\partial x^{a_{i}}} \prod_{j=1}^{\beta}\frac{\partial^{b_j+c_j} e^n}{\partial x^{b_j} \partial y^{c_j}}\bigg\| &
= & \bigg\|\frac{\partial^\alpha e_{n+1}}{\partial x^\alpha}\circ(f^n, e^n). \prod_{i=1}^{\alpha} \partial^{a_i} f^n\bigg\| \\
& < & \bfC C_1(r,\bfC)^\alpha \mu^{r.n}  \lambda^n.
\end{eqnarray*}


\vskip 5pt

\noindent {\it Case $\beta=1$ and $\alpha=0$.} From the fact $|\tfrac {\partial e}{\partial y}|\leq \lambda $ and the induction (III), it follows that 
\begin{eqnarray*}
\bigg\|\frac{\partial^{\alpha+\beta} e_{n+1}}{\partial x^{\alpha} \partial y^{\beta}}(f^n,e^n)\prod_{i=1}^{\alpha} \frac{\partial^{a_i} f^n}{\partial x^{a_{i}}} \prod_{j=1}^{\beta}\frac{\partial^{b_j+c_j} e^n}{\partial x^{b_j} \partial y^{c_j}}\bigg\| & = &
\bigg\|\frac{\partial e_{n+1}}{\partial y}.\frac{\partial^{r} e^n}{\partial x^{r_1}\partial y^{r_2}}\bigg\| \\
&<& \lambda C_2(r,\bfC) \mu^{r.n}  \lambda^n.
\end{eqnarray*}


\vskip 5pt




\noindent {\it Case $\beta\geq 1$ and $\alpha+\beta >1$.}
Note that $b_j+c_j<r$. From (II) and the induction (III):
\begin{eqnarray*}
& &\bigg\|\frac{\partial^{\alpha+\beta} e_n}{\partial x^{\alpha} \partial y^{\beta}}(f^n,e^n)\prod_{i=1}^{\alpha} \frac{\partial^{a_i} f^n}{\partial x^{a_{i}}} \prod_{j=1}^{\beta}\frac{\partial^{b_j+c_j} e^n}{\partial x^{b_j} \partial y^{c_j}}\bigg\|\\
& & < \bfC.C_1(r,\bfC)^\alpha.\mu^{A_\alpha\cdot n} C_2(r-1,\bfC)^\beta \mu^{B_\alpha.(n-1)}  \lambda^{\beta(n-1)}\\
& & \leq \bfC.C_1(r,\bfC)^r C_2(r-1,\bfC)^r \mu^{r\cdot n}.  \lambda^n.
\end{eqnarray*}

\vskip 5pt



To conclude, we have to choose the constant $C_2(r,\bfC)$ inductively on $r$ (the order of derivations), so that the sum of the upper bounds over all the terms
in Fa\'a di Bruno formula is smaller than $C_2(r,\bfC) \mu^{r.n}  \lambda^n$. More precisely, given $C_2(r-1, l)$ we have to choose $C_2(r, \bfC)$ such that 
\begin{equation}\label{final}\lambda C_2(r,\bfC) + M. \bfC. C_1(r,\bfC)^r + N \bfC.C_1(r,\bfC)^r C_2(r-1,\bfC)^r < C_2(r,\bfC),
\end{equation}
where $N$ is the number of  factors of the case $\beta=0,\alpha>0$ and $M$ is the number of  factors of the case $\beta\geq 1,\alpha+\beta>1.$ To do that, recall that $C_1(r, \bfC)$ was already chosen and so inequality \ref{final} it is equivalent to chose $C_2(r, \bfC)$  such that 
$$\frac{1}{1-\lambda}[M. \bfC. C_1(r,\bfC)^r + N. \bfC.C_1(r,\bfC)^r C_2(r-1,\bfC)^r] < C_2(r,\bfC)$$ and this 
can be obtained  since $\lambda<1$. 
\end{proof}


\bigskip

\begin{tabular}{l l l}
\emph{Sylvain Crovisier} &&
\emph{Mikhail Lyubich}\\

Laboratoire de Math\'ematiques d'Orsay &&
Institute for Mathematical Sciences\\

CNRS - Univ. Paris-Saclay &&
Stony Brook University\\

Orsay, France &&
Stony Brook, NY, USA\\

&&\\

\emph{Enrique Pujals} &&
\emph{Jonguk Yang}\\

Graduate Center - CUNY &&
Institut für Mathematik\\

New York, USA &&
Universität Zürich\\

&& Zürich, Switzerland

\end{tabular}

\end{document}